\documentclass[12pt]{amsart}
\usepackage[colorlinks=true,citecolor=black,linkcolor=black,urlcolor=blue]{hyperref}
\usepackage{amsmath}
\usepackage[english,  activeacute]{babel}
\usepackage[utf8]{inputenc}
\usepackage{amssymb}
\usepackage{amsthm}
\usepackage{graphics,graphicx}
\usepackage{enumerate}
\usepackage{ytableau}
\usepackage{array}
\usepackage{bm}
\usepackage{cite}
\usepackage{a4wide}
\usepackage{color, url}
\usepackage{float}

\usepackage{comment}
\usepackage{xcolor}
\usepackage{color, url}
\usepackage{float}
\usepackage{accents}
\usepackage{pgf, tikz, forest}
\usepackage{multirow}

\usepackage{tikz}
\usetikzlibrary{decorations.pathreplacing}

\usepackage{booktabs}
\usepackage{tabularx}
\usepackage{geometry}
\theoremstyle{plain}
\newtheorem{theorem}{Theorem}[section]
\newtheorem{proposition}[theorem]{Proposition}
\newtheorem{coro}[theorem]{Corollary}
\newtheorem{lemma}[theorem]{Lemma}

\theoremstyle{definition}

\newcommand{\Z}{{\mathbb Z}}

\DeclareMathOperator{\des}{des}
\newcommand{\exc}{{\texttt{exc}}}

\newcommand{\highlightmin}[1]{%
  \begingroup
  \setlength{\fboxsep}{2pt}%
  \colorbox{red!20}{\ensuremath{#1}}%
  \endgroup
}

\usepackage{wasysym}

\DeclareMathOperator{\lv}{lv}

\newcommand{\F}{{\mathcal F}}

\DeclareMathOperator{\pk}{pk}
\DeclareMathOperator{\ret}{ret}
\DeclareMathOperator{\lev}{lev}
\DeclareMathOperator{\asc}{asc}
\DeclareMathOperator{\rlmin}{rlmin}
\DeclareMathOperator{\lmin}{lmin}
\DeclareMathOperator{\yper}{yper}
\DeclareMathOperator{\xper}{xper}
\DeclareMathOperator{\sper}{sper}
\DeclareMathOperator{\area}{area}

\usepackage[usestackEOL]{stackengine}[2013-10-15]
\def\x{\hspace{3ex}}    %BETWEEN TWO 1-DIGIT NUMBERS
\def\y{\hspace{2.45ex}}  %BETWEEN 1 AND 2 DIGIT NUMBERS
\stackMath

\title{Fibonacci Polyominoes: Refined Enumeration and Dyck-Path Bijections}

\author[J.-L. Baril]{Jean-Luc Baril}
\address{LIB, Universit\'e Bourgogne Europe,
  B.P. 47 870, 21078 Dijon Cedex France}
\email{barjl@u-bourgogne.fr}

\author[J.L. Ram\'irez]{Jos\'e L. Ram\'irez}
\address{Departamento de Matem\'aticas,  Universidad Nacional de Colombia,  Bogot\'a, Colombia}
\email{jlramirezr@unal.edu.co}

\author[S. Ram\'irez]{Samuel Ram\'irez}
\address{Departamento de Matem\'aticas,  Universidad Nacional de Colombia,  Bogot\'a, Colombia}
\email{samramirezra@unal.edu.co}

\date{\today}

\begin{document}

\newcommand{\nadji}[1]{\mbox{}{\sf\color{green}[Ram\'{\i}rez: #1]}\marginpar{\color{green}\Large$*$}}

\begin{abstract}
We study Fibonacci polyominoes, a class of column-convex polyominoes whose lower boundary is a staircase path made of unit horizontal and vertical steps. We derive generating functions that enumerate these polyominoes by area and semiperimeter. The resulting formulas give refinements of the classical Fibonacci enumeration and lead to explicit expressions involving Catalan numbers. We give a bijection with labeled Dyck paths that provides combinatorial proofs of these
formulas and translates natural statistics on Fibonacci polyominoes into peaks, returns, and related statistics on Dyck paths. This yields
refinements by descents and left-to-right minima, including distributions governed by Narayana numbers. We also study consecutive columns of equal height, leading in the diagonal case to a refinement of the Catalan enumeration governed by Motzkin numbers.
\end{abstract}

\maketitle

\begin{center}
\emph{Dedicated to Elena Barcucci\footnote{On the occasion of her retirement, in recognition of her contributions to enumerative combinatorics and the study of polyominoes.}}
\end{center}

\section{Introduction}

The enumeration of polyomino classes defined by convexity and directionality constraints is a classical topic in enumerative combinatorics; see \cite{Book1} and the references therein. A \emph{polyomino} is a finite union of unit squares in the grid $\Z^2$ that is connected by edge adjacency. A polyomino is \emph{column-convex} if its intersection with any vertical line is either empty or a single contiguous segment of cells. Several statistics on polyominoes, including area and perimeter, have been studied for a variety of families; see, for example, \cite{BleBreKnop3,Bou,ManSha2}.

An important family in this setting is that of \emph{directed column-convex polyominoes}, or dcc-polyominoes. A column-convex polyomino is directed if there is a source cell from which every other cell can be reached using only north and east steps. Barcucci, Del Lungo, Fezzi, and Pinzani showed that the number of dcc-polyominoes of area $n$ is the Fibonacci number $F_{2n-1}$. These objects are also in bijection with nondecreasing Dyck paths of semilength $n$; see \cite{BarDelFezPin, DeutschP, FlorezJoseJunes}.

The polyominoes considered in the present paper form a natural subfamily of the dcc-polyominoes. A \emph{Fibonacci polyomino} is a column-convex polyomino whose lower boundary is a staircase
path made of unit horizontal and vertical steps. Equivalently, when its columns are read from left to right, the bottom cell of each column lies one unit above that of the preceding column. In particular, every Fibonacci polyomino is directed, with its lower-left cell as source. Figure~\ref{Ex1} shows a Fibonacci polyomino with five columns and area~18.

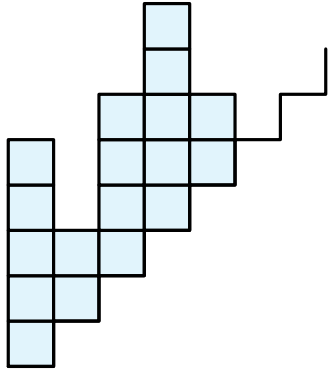
\begin{figure}[ht!]
\centering
\centering
\begin{tikzpicture}[scale=0.6, line cap=round, line join=round]
\tikzset{cell/.style={draw=black, very thick, fill=cyan!10!white}}

% Celdas
\foreach \x/\y in {0/0,0/1,0/2,0/3,0/4, 1/1, 1/2, 2/2, 2/3, 2/4, 2/5, 3/3, 3/4, 3/5, 3/6, 3/7,4/4, 4/5}{
  \draw[cell] (\x,\y) rectangle ++(1,1);
}

  \draw[very thick]
    (0,0) -- (1,0) -- (1,1) -- (2,1) --(2,2)--(3,2)--(3,3)--(4,3)--(4,4)--(5,4)--(5,5)--(6,5)--(6,6)--(7,6)--(7,7);

\end{tikzpicture}

\caption{Fibonacci polyomino of area 18.}\label{Ex1}
\end{figure}

Fibonacci polyominoes were studied by Turban \cite{Turban1,Turban2} in connection with lattice animals on a staircase. He proved that the number of Fibonacci polyominoes with $n$ cells is $F_n$, where $F_n$ denotes the $n$th Fibonacci number, and considered more general models in which the staircase steps may have arbitrary heights. Thus, within the larger class of dcc-polyominoes, the staircase condition leads to the ordinary Fibonacci enumeration $F_n$, rather than the odd-indexed Fibonacci enumeration $F_{2n-1}$.

Our aim is to refine this Fibonacci enumeration and to explain the Catalan structures that arise when perimeter statistics are taken into
account. We first derive generating functions that record area and the horizontal and vertical semiperimeters. Their specializations recover
the Fibonacci enumeration by area and lead to explicit formulas involving Catalan numbers for the distribution of the two semiperimeters.

The main combinatorial tool of the paper is a bijection between certain sequences of nonnegative integers and Dyck paths whose up steps carry
nonnegative integer labels. When applied to Fibonacci polyominoes, this correspondence gives a bijective explanation of the Catalan formulas
obtained from the generating functions. More importantly, it translates natural statistics on the sequence of column heights into classical and
refined statistics on Dyck paths. In particular, descents and left-to-right minima correspond to peaks and returns, leading to refinements involving Narayana numbers and the joint distribution of these two Dyck-path statistics. A further statistic recording consecutive columns of equal height is translated into a statistic on
certain factors of Dyck paths. In the diagonal case, this correspondence leads to a refinement of the Catalan enumeration governed by Motzkin
numbers.

A preliminary version of part of this work appeared in the proceedings of GASCom 2026 \cite{BarilRamirezRamirezVillamizarGASCom}. The present paper substantially extends that work through the labeled Dyck-path bijection and the refined distributions developed here.

The rest of the paper is organized as follows. Section~2 develops the generating-function approach and studies Fibonacci polyominoes by area
and horizontal and vertical semiperimeters. In Section~3, we construct
the bijection with labeled Dyck paths and use it to give a combinatorial
interpretation of the Catalan formulas. Section~4 exploits this
bijection to obtain refinements by descents, left-to-right minima, and
levels, including distributions involving Narayana and Motzkin
numbers. Finally, Section~5 contains concluding remarks and discusses
possible extensions to related staircase models.

\section{A Generating Function Approach to Perimeter and Area}

Let $P$ be a Fibonacci polyomino. Its \emph{area}, denoted by $\area(P)$, is the  number of cells of $P$. The \emph{perimeter} of $P$ is the number of unit edges on its boundary, and its \emph{semiperimeter}, denoted by $\sper(P)$, is half of this number. We refine the semiperimeter according to the orientation of the boundary edges. The  \emph{horizontal semiperimeter} $\xper(P)$ is half the number of horizontal boundary edges, while the \emph{vertical semiperimeter} $\yper(P)$ is half the number of vertical boundary edges. Thus, $\sper(P)=\xper(P)+\yper(P)$. For the Fibonacci polyomino in Figure~\ref{Ex1}, we have \[
\area(P)=18,\qquad
\xper(P)=5,\qquad
\yper(P)=10,\qquad
\sper(P)=15.
\]

Let $\F$ denote the set of all Fibonacci polyominoes. For $m,n,k\in\Z^{+}$, let $\F_{m,n,k}$ be the set of
Fibonacci polyominoes with $2m$ horizontal boundary edges, $2n$ vertical boundary edges, and area $k$. For $m,n\geq1$, we also set
\[
\F_{m,n}
:=
\bigcup_{k\geq1}\F_{m,n,k}.
\]
Thus, $\F_{m,n}$ is the set of Fibonacci polyominoes with horizontal semiperimeter $m$ and vertical semiperimeter $n$, with no
restriction on the area. Consequently,
\[
\F
=
\bigcup_{m,n\geq1}\F_{m,n}
=
\bigcup_{m,n,k\geq1}\F_{m,n,k}.
\]

We consider the generating function
\[
F(x,y,q)=\sum_{P\in \F}x^{\xper(P)}y^{\yper(P)}q^{\area(P)}=\sum_{m,n,k\geq 1} \sum_{P\in \F_{m,n,k}}x^{m}y^{n}q^{k}.
\]
For $h\geq 1$, let $F_h(x,y,q)$ be the generating function for Fibonacci polyominoes whose first column has height $h$. Thus,  \[
F(x,y,q)=\sum_{h\geq 1} F_h(x,y,q).
\]
To keep track of the height of the first column, introduce a catalytic variable $s$ and define
\[
F(x,y,q;s)
=
\sum_{h\geq 1}F_h(x,y,q)s^h.
\]
In particular, $F(x,y,q;1)=F(x,y,q)$. For brevity, we write $F(s)$ for $F(x,y,q;s)$ and $F(1)$ for $F(x,y,q)$.

We next derive a functional equation for $F(s)$ and use it to determine $F(1)$.

\begin{theorem}\label{pol:areaper}
The generating function for nonempty Fibonacci polyominoes, counted by horizontal semiperimeter, vertical semiperimeter, and area, is
\[
F(x,y,q)
=
\frac{
\sum_{\ell\geq 0}
\frac{(xy)^{\ell+1}q^{(\ell+1)^2}(y-1)^\ell}
     {1-yq^{\ell+1}}
\prod_{i=0}^{\ell-1}
\frac{1}{(1-yq^{i+1})(1-q^{i+1})}
}{
1-
\sum_{\ell\geq 0}
\frac{(xy)^{\ell+1}q^{(\ell+1)(\ell+2)}(y-1)^\ell}
     {1-q^{\ell+1}}
\prod_{i=0}^{\ell-1}
\frac{1}{(1-yq^{i+1})(1-q^{i+1})}
}.
\]
\end{theorem}

\begin{proof}
We first consider Fibonacci polyominoes consisting of a single column. If this column has height $h\geq 1$, then $\xper(P)=1$, $\yper(P)=h$, and $\area(P)=h$. Its contribution  is therefore $xy^hq^hs^h$. Summing over all $h\geq 1$ gives 
\[x\sum_{h\geq 1}(yqs)^{h}=\frac{xyqs}{1-yqs}.\]

Now suppose that the polyomino has at least two columns. Let $i$ and $h$ be the heights of its first and second columns, respectively. Since the lower boundary is a staircase with unit vertical steps, the second column begins one unit above the first. In particular, $i\geq 2$. There are two cases.

Case 1: $2\leq i\leq h+1$.  Start with a Fibonacci polyomino counted by $F_h(x,y,q)$ and prepend a column of height $i$; see Figure~\ref{case1th1}\,(left). The new column contributes  $(sq)^i$ to record its height and area. Prepending it increases both the horizontal and vertical semiperimeters by one, giving an additional factor  $xy$. Hence the contribution of this case is
\begin{align*}
\sum_{h\geq 1} xy \sum_{i=2}^{h+1} (sq)^{i} F_h(x,y,q)
&= xy (sq)^{2} \sum_{h\geq 1} \frac{1-(sq)^{h}}{1-sq} F_h(x,y,q)\\
&= \frac{xy (sq)^{2}}{1-sq} \big(F(1)-F(sq)\big).
\end{align*}

\begin{figure}[ht!]
\centering
\begin{tikzpicture}[scale=0.6, line cap=round, line join=round]
  % --- parámetros: 2 <= i <= h+1 ---
  \def\i{3}  % altura de la primera columna
  \def\h{5}  % altura de la segunda columna

  %----- Primera columna (altura i) -----
  \draw[fill=cyan!10!white] (-3,0) rectangle (-2,\i);
  \foreach \y in {1,...,\numexpr\i-1} { \draw[very thick] (-3,\y) -- (-2,\y); }
  \draw[very thick] (-3,0) rectangle (-2,\i);

  %----- Segunda columna (altura h) -----
  \draw[fill=white] (-2,1) rectangle (-1,\h+1);
  \foreach \y in {1,...,\numexpr\h} { \draw (-2,\y) -- (-1,\y); }
  \draw[very thick] (-2,1) rectangle (-1,\h+1);

  %----- Escalera (pasos unitarios) a la derecha -----
  \draw[very thick]
    (-3,0) -- (-2,0) -- (-2,1) -- (-1,1) -- (-1,2) -- (0,2)
    -- (0,3) -- (1,3) -- (1,4) -- (2,4) -- (2,5);

  %----- Indicador de "continúa el poliminó" a partir de la 2ª columna -----
  % Zona fantasma (placeholder) donde puede seguir el poliminó (contado por F_h)
  \draw[-,dashed, thick] (-1,2) -- (-1,7)--(0,7)--(0,6)--(2,6)--(2,5);

  %----- Marcadores de altura i y h -----
  \draw[<->] (-3.35,0) -- node[left] {$i$} (-3.35,\i);
  \draw[<->] (-1.65,1) -- node[right] {$h$} (-1.65,\h+1);
\end{tikzpicture} \hspace{2cm}
\begin{tikzpicture}[scale=0.6, line cap=round, line join=round]
  % --- parámetros: 2 <= i <= h+1 ---
  \def\i{8}  % altura de la primera columna
  \def\h{5}  % altura de la segunda columna

  %----- Primera columna (altura i) -----
  \draw[fill=cyan!10!white] (-3,0) rectangle (-2,\i);
  \foreach \y in {1,...,\numexpr\i-1} { \draw[very thick] (-3,\y) -- (-2,\y); }

  %----- Parte duplicada: del segundo al quinto cuadro -----
  \begin{scope}
    \clip (-3,1) rectangle (-2,6);
    \foreach \t in {-6,-5.9,...,6} {
      \draw[gray!70, thin] (-3+\t,1) -- (-2+\t,6);
    }
  \end{scope}

  \draw[very thick] (-3,0) rectangle (-2,\i);

  %----- Segunda columna (altura h) -----
  \draw[fill=white] (-2,1) rectangle (-1,\h+1);
  \foreach \y in {1,...,\numexpr\h} { \draw (-2,\y) -- (-1,\y); }
  \draw[very thick] (-2,1) rectangle (-1,\h+1);

  %----- Escalera (pasos unitarios) a la derecha -----
  \draw[very thick]
    (-3,0) -- (-2,0) -- (-2,1) -- (-1,1) -- (-1,2) -- (0,2)
    -- (0,3) -- (1,3) -- (1,4) -- (2,4) -- (2,5);

  %----- Indicador de "continúa el poliminó" a partir de la 2ª columna -----
  \draw[-,dashed, thick] (-1,2) -- (-1,7)--(0,7)--(0,6)--(2,6)--(2,5);

  %----- Marcadores de altura i y h -----
  \draw[<->] (-3.35,0) -- node[left] {$i$} (-3.35,\i);
  \draw[<->] (-1.65,1) -- node[right] {$h$} (-1.65,\h+1);
\end{tikzpicture}
\caption{Decomposition according to the relative heights of the first two columns: Case~1 (left) and Case~2 (right).}\label{case1th1}
\end{figure}
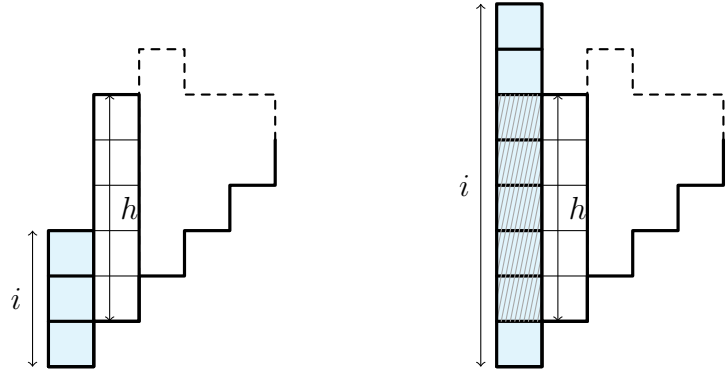

Case 2: $i\geq h+2$. Start with a Fibonacci polyomino whose first column has height $h$ and duplicate this column to its left; see Figure~\ref{case1th1}\,(right). Since the duplicated column contributes $h$ new cells, this operation is encoded by replacing $s$ with $sq$ in $F(s)$. It therefore contributes  $F(sq)$. We next add one cell below the duplicated column. This  increases both semiperimeters by one, giving the factor $sxyq$. Since $i\geq h+2$, at least one additional cell must be placed above the duplicated column. These cells  contribute the geometric factor $ysq/(1-ysq)$.  Thus the contribution of Case~2 is
\[
xyqs\cdot \frac{yqs}{1 - yqs} \cdot F(sq)
= F(s q)\frac{x(yqs)^{2}}{1 - yqs}.
\]
Combining the single-column case with Cases~1 and~2 gives
\begin{equation}\label{eq:functionalareaper}
F(s)
= \frac{xyqs}{1-yqs} + \frac{xy (qs)^{2}}{1-qs} F(1)
+ \left( \frac{x (yqs)^{2}}{1-yqs} - \frac{xy(qs)^{2}}{1-qs} \right) F(sq).
\end{equation}
Set
\[
A(s)= \frac{xyqs}{1-yqs},\quad
B(s)= \frac{xy (qs)^{2}}{1-qs},\quad \text{and} \quad 
C(s)= \frac{x (yqs)^{2}}{1-yqs} - \frac{xy(qs)^{2}}{1-qs}=\frac{xyq^2s^2(y-1)}{(1-syq)(1-sq)}.
\]
Equation~\eqref{eq:functionalareaper} becomes
\[
F(s)=A(s)+B(s)F(1)+C(s)F(sq).
\]
Iterating this relation  and  assuming that $|q|\leq 1$ and $|s|\leq 1$,  we obtain
\[
F(s)
= \sum_{\ell\geq 0} A(s q^{\ell}) \prod_{i=0}^{\ell-1} C(s q^{i})
+ \left( \sum_{\ell\geq 0} B(s q^{\ell}) \prod_{i=0}^{\ell-1} C(s q^{i}) \right) F(1),
\]
where an empty product is equal to $1$. Setting $s=1$ and solving for $F(1)$ gives
\[
F(1)
= \frac{\sum_{\ell\geq 0} A(q^{\ell}) \prod_{i=0}^{\ell-1} C(q^{i})}
       {1 - \sum_{\ell\geq 0} B(q^{\ell}) \prod_{i=0}^{\ell-1} C(q^{i})}.
\]
Substituting the explicit expressions for $A$, $B$, and $C$ and simplifying gives the stated formula.
\end{proof}

The first few terms in the expansion of $F(x,y,q)$ with respect to the area variable $q$ are
\begin{multline*}
    F(x,y,q)=x y q + x y^2 q^2 + (x^2 y^2 + x y^3) q^3 + (2 x^2 y^3 + 
    x y^4) q^4 \\+ \bm{(x^2 y^3 + x^3 y^3 + 2 x^2 y^4 + 
    x y^5) q^5} + (x^3 y^3 + 2 x^2 y^4 + 2 x^3 y^4 + 2 x^2 y^5 + 
    x y^6) q^6 + O(q^7).
\end{multline*}
The coefficient of $q^5$, displayed in bold, is illustrated in
Figure~\ref{fig1}.

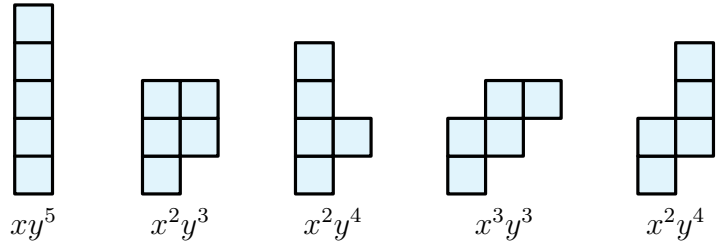
\begin{figure}[ht!]
\centering
\begin{tikzpicture}[scale=0.5, line cap=round, line join=round]
  \draw[fill=cyan!10!white] (-3,0) rectangle (-2,5);
  \foreach \y in {1,2,3,4}
    \draw[very thick] (-3,\y) -- (-2,\y);
  \draw[very thick] (-3,0) rectangle (-2,5);
  \node at (-2.5,-0.8) {$xy^5$};
\end{tikzpicture}
\qquad
\begin{tikzpicture}[scale=0.5, line cap=round, line join=round]
  \foreach \x/\y in {0/2,0/3,1/3,0/4,1/4}{
    \draw[fill=cyan!10!white] (\x-3,\y-2) rectangle ++(1,1);
    \draw[very thick] (\x-3,\y-2) rectangle ++(1,1);
  }
  \node at (-2,-0.8) {$x^2y^3$};
\end{tikzpicture}
\qquad
\begin{tikzpicture}[scale=0.5, line cap=round, line join=round]
  \foreach \x/\y in {0/2,0/3,0/4,0/5,1/3}{
    \draw[fill=cyan!10!white] (\x-3,\y-2) rectangle ++(1,1);
    \draw[very thick] (\x-3,\y-2) rectangle ++(1,1);
  }
  \node at (-2,-0.8) {$x^2y^4$};
\end{tikzpicture}
\qquad
\begin{tikzpicture}[scale=0.5, line cap=round, line join=round]
  \foreach \x/\y in {0/2,0/3,1/3,1/4,2/4}{
    \draw[fill=cyan!10!white] (\x-3,\y-2) rectangle ++(1,1);
    \draw[very thick] (\x-3,\y-2) rectangle ++(1,1);
  }
  \node at (-1.5,-0.8) {$x^3y^3$};
\end{tikzpicture}
\qquad
\begin{tikzpicture}[scale=0.5, line cap=round, line join=round]
  \foreach \x/\y in {0/2,0/3,1/3,1/4,1/5}{
    \draw[fill=cyan!10!white] (\x-3,\y-2) rectangle ++(1,1);
    \draw[very thick] (\x-3,\y-2) rectangle ++(1,1);
  }
  \node at (-2,-0.8) {$x^2y^4$};
\end{tikzpicture}
\caption{The five Fibonacci polyominoes of area $5$, together with their semiperimeter weights.}
\label{fig1}
\end{figure}

The specialization retaining only the area and the height of the first column has a particularly simple form. Setting $x=y=1$ in
\eqref{eq:functionalareaper} gives the following result.

\begin{coro}\label{coroteo1}
The generating function for nonempty Fibonacci polyominoes, counted by area and the height of the first column, is
\[
H(q,s):=
F(1,1,q;s)
=
\frac{qs\bigl(1-q-q^2(1-s)\bigr)}
     {(1-q-q^2)(1-qs)}.
\]
Moreover, for $m\geq 0$ and $n\geq m+3$, $[q^ns^{m+2}]H(q,s)=F_{n-m-2}$, while $[q^{m+2}s^{m+2}]H(q,s)=1$.
\end{coro}

By calculating $\partial_sH(q,s)\vert_{s=1}$, and applying
standard coefficient asymptotics \cite{Fla,Orl} yields the limiting mean height of the first column.

\begin{coro}\label{coroasympt} An asymptotic approximation for the average of the number of cells in the first column among all Fibonacci polyominoes of a given area is  $$\frac{5+\sqrt{5}}{2}.$$
\end{coro}

\subsection{The area}

We first specialize the functional equation by ignoring the vertical
semiperimeter. Setting $s=1$ and $y=1$ in
\eqref{eq:functionalareaper} gives 
\[F(x,q)=\frac{qx}{1-q}+ \frac{xq^2}{1-q}F(x,q),\]
where $F(x,q):=F(x,1,q)$. Let $a(n,m)$ denote the number of Fibonacci polyominoes of area $n$ and horizontal semiperimeter $m$, equivalently, with $m$ columns. The following result recovers Turban’s enumeration by area and number of occupied stairs. We include a direct bijective proof. 

\begin{theorem}
The generating function for nonempty Fibonacci polyominoes, counted by area and horizontal semiperimeter, is
\[
F(x,q)=\sum_{n,m\geq 1} a(n,m)\,x^{m}q^{n}
     = \frac{qx}{1 - q - q^{2}x}.
\]
Moreover, for $n, m\geq 1$, $a(n,m)=\binom{n-m}{m-1}$.
\end{theorem}

A direct bijective proof of the formula $a(n,m)=\binom{n-m}{m-1}$ is also possible. Let $P$ be a Fibonacci polyomino of area $n$ and horizontal semiperimeter $m$ (so $P$ has $m$ columns). For $1\leq i\leq m$, let $h_i$ be the height of the $i$th column. By the staircase condition,  $h_i\geq 2$ for $1\leq i\leq m-1$, and $h_m\geq 1$.  Associate with $P$ the north-east lattice path
\[N^{h_1-2}EN^{h_2-2}E \cdots N^{h_{m-1}-2}EN^{h_{m}-1},\] where $N=(0,1)$ and $E=(1,0)$. This is a path from the origin $(0,0)$ to  $(m-1,n-2m+1)$, with exactly $m-1$ east steps and $n-2m+1$ north steps. Hence the number of such paths is $\binom{(m-1)+(n-2m+1)}{m-1}=\binom{n-m}{m-1}$. This construction is a bijection.

For example, the Fibonacci polyomino in Figure~\ref{Ex1} has column heights
$(h_1,h_2,h_3,h_4,h_5)=(5,2,4,5,2)$. Its corresponding lattice path is \[N^{5-2}EN^{2-2}EN^{4-2}EN^{5-2}EN^{2-1}=N^{3}E^2N^{2}EN^{3}EN.\] 
As expected, this path ends at $(5-1,\,18-2\cdot 5+1)=(4,9)$.

Specializing $x=1$ in the previous theorem yields the generating function $q/(1-q-q^2)$ for the Fibonacci numbers, as expected. We next describe a bijection between Fibonacci polyominoes of area $n$ and compositions of $n-1$ with parts $1$ and $2$, another classical family counted by the Fibonacci numbers; see, for example, \cite{HeubachMansour}.

Let $\alpha=(\alpha_1,\dots,\alpha_r)$ be a composition of $n-1$ with
$\alpha_i\in\{1,2\}$ for all $i$. We construct a Fibonacci polyomino by
reading $\alpha$ from left to right. Begin with a single cell in the first
column and designate it as the \emph{current} cell. If $\alpha_i=1$, then
add one cell immediately above the current cell, so that we remain in the
same column, and declare this new cell to be the current one. If
$\alpha_i=2$, then add a vertical domino whose lower cell is the current
cell, move to the next column to the east, and declare the bottom cell of
this new column to be the current one. After all parts have been processed,
add one final cell, marked $\mathbf{x}$, at the last current position, using
the same rule. In this way we obtain a well-defined Fibonacci polyomino of
area $n$. Moreover, different compositions clearly produce different
polyominoes.

Conversely, let $P$ be a Fibonacci polyomino of area $n$. We recover the
composition $\alpha$ by reading the columns of $P$ from left to right. If $P$
consists of a single column of height $n$, then $\alpha=1^{n-1}$, and the
marked cell $\mathbf{x}$ is the top cell. Otherwise, each non-last column of
height $h$ contributes the block $1^{h-2}2$, while the last column of height
$h$ contributes $1^{h-1}$, followed by the terminal marker $\mathbf{x}$. This
reconstruction is unique and clearly inverts the previous construction.
Hence the correspondence is bijective. See Figure~\ref{fig2b} for an
illustration.

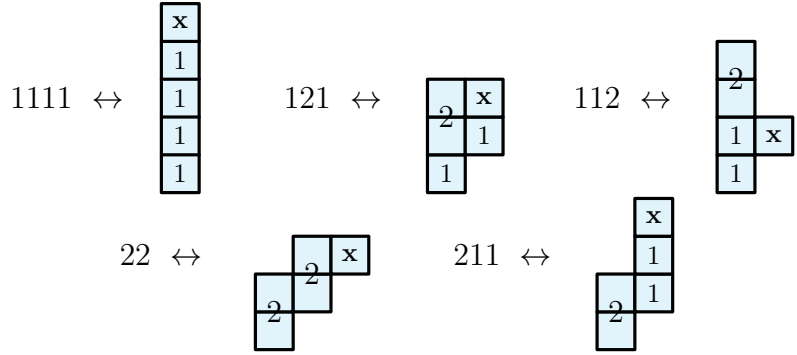
\begin{figure}[ht!]
\centering
\begin{tikzpicture}[scale=0.5, line cap=round, line join=round]
  \draw[fill=cyan!10!white] (-3,0) rectangle (-2,5);
  \foreach \y in {1,2,3,4} { \draw[very thick] (-3,\y) -- (-2,\y); }
  \draw[very thick] (-3,0) rectangle (-2,5);

  \node at (-2.5,0.5) {\footnotesize $1$};
  \node at (-2.5,1.5) {\footnotesize $1$};
  \node at (-2.5,2.5) {\footnotesize $1$};
  \node at (-2.5,3.5) {\footnotesize $1$};
  \node at (-2.5,4.5) {\footnotesize $\mathbf{x}$};
  \node at (-5.5,2.5) {$1111\ \leftrightarrow$};
\end{tikzpicture}
\qquad
\begin{tikzpicture}[scale=0.5, line cap=round, line join=round]
\foreach \x/\y in {0/2, 0/3, 1/3, 0/4, 1/4}{
  \draw[fill=cyan!10!white] (\x-3,\y-2) rectangle ++(1,1);
  \draw[very thick] (\x-3,\y-2) rectangle ++(1,1);
}

\node at (-2.5,0.5) {\footnotesize $1$};
\node at (-2.5,2) {$2$};
\node at (-5.5,2.5) {$121\ \leftrightarrow$};

\node at (-1.5,1.5) {\footnotesize $1$};
\node at (-1.5,2.5) {\footnotesize $\mathbf{x}$};
\end{tikzpicture}
\hspace{0.5cm}
\begin{tikzpicture}[scale=0.5, line cap=round, line join=round]
  \foreach \x/\y in {0/2, 0/3, 0/4, 0/5, 1/3}{
    \draw[fill=cyan!10!white] (\x-3,\y-2) rectangle ++(1,1);
    \draw[very thick] (\x-3,\y-2) rectangle ++(1,1);
  }

\node at (-2.5,0.5) {\footnotesize $1$};
\node at (-2.5,1.5) {\footnotesize $1$};
\node at (-2.5,3) {$2$};
\node at (-5.5,2.5) {$112\ \leftrightarrow$};

\node at (-1.5,1.5) {\footnotesize $\mathbf{x}$};
\end{tikzpicture}

\begin{tikzpicture}[scale=0.5, line cap=round, line join=round]
  \foreach \x/\y in {0/2, 0/3, 1/3, 1/4, 2/4}{
    \draw[fill=cyan!10!white] (\x-3,\y-2) rectangle ++(1,1);
    \draw[very thick] (\x-3,\y-2) rectangle ++(1,1);
  }

\node at (-2.5,1) {$2$};
\node at (-5.5,2.5) {$22\ \leftrightarrow$};

\node at (-1.5,2) {$2$};
\node at (-0.5,2.5) {\footnotesize $\mathbf{x}$};
\end{tikzpicture}
\qquad
\begin{tikzpicture}[scale=0.5, line cap=round, line join=round]
  \foreach \x/\y in {0/2, 0/3, 1/3, 1/4, 1/5}{
    \draw[fill=cyan!10!white] (\x-3,\y-2) rectangle ++(1,1);
    \draw[very thick] (\x-3,\y-2) rectangle ++(1,1);
  }

\node at (-2.5,1) {$2$};
\node at (-5.5,2.5) {$211\ \leftrightarrow$};

\node at (-1.5,1.5) {\footnotesize $1$};
\node at (-1.5,2.5) {\footnotesize $1$};
\node at (-1.5,3.5) {\footnotesize $\mathbf{x}$};
\end{tikzpicture}
\caption{The bijection between compositions of $4$ and Fibonacci polyominoes
of area $5$.}
\label{fig2b}
\end{figure}

We next study the total semiperimeter of Fibonacci polyominoes of fixed area.

\begin{coro}
Let $S_n$ be the total semiperimeter of all Fibonacci polyominoes of area $n$. Then 
\begin{align*}\sum_{n\geq 1} S_n q^n&=\frac{q (2 - 3 q + q^2 - q^4)}{(1 - q) (1 - q - q^2)^2}.\end{align*}
Moreover, for $n\geq 1$ we have
\[
S_n =  \frac{2}{5}(n+6)F_n + \frac{n}{5}F_{n+1}-1.
\]
An asymptotic approximation for the average of the semiperimeter in all Fibonacci polyominoes of a given area  is  
\[\left(\frac{5+\sqrt{5}}{10}\right)n.\]
\end{coro}
\begin{proof}
By setting $y=x$ and $s=1$ in \eqref{eq:functionalareaper} gives
\begin{align}\label{funceccoro25}
    F(1)
= \frac{x^2q}{1-xq}
+ \frac{x^2q^{2}}{1-q} F(1)
+ \left( \frac{x^3 q^{2}}{1-xq} - \frac{x^2q^2}{1-q} \right) F(q).
\end{align}
Define $B(q):=\frac{\partial}{\partial x}F(x,x,q;1)\rvert_{x=1}$.  Since $x$ marks the horizontal semiperimeter and we  set $y=x$, the derivative at $x=1$ records the total value of $\xper(P)+\yper(P)=\sper(P)$, hence
\[
B(q)=\sum_{n\geq 1} S_n q^n.
\]
Differentiating the identity \eqref{funceccoro25} with respect to $x$ and then setting $x=1$, we obtain
\[B(q)=\frac{(2 - q) q}{(1-q)^2} + \frac{2 q^2}{1-q}F(1,1,q) + \frac{q^2}{1-q}B(q) +\frac{q^2}{(1-q)^2} H(q,q),  \]
where $F(1,1,q)=q/(1-q-q^2)$ and  $H(q,q)=F(1,1,q;q)$ is the generating function given in Corollary~\ref{coroteo1}, namely
$$H(q,q)=\frac{(1 - q) q^2}{1 - q - q^2}.$$
Solving for $B(q)$ and substituting the above expressions for $F(1,1,q)$ and $H(q,q)$ yields the desired result. 

Finally, the closed form for $S_n$ follows by extracting coefficients from this rational function and using standard identities for the Fibonacci numbers.
\end{proof}

\subsection{The perimeter}
We now ignore the area statistic by setting $q=1$ in
\eqref{eq:functionalareaper}. This yields a bivariate generating function for Fibonacci polyominoes counted by horizontal and vertical
semiperimeter. We compute this generating function by the kernel method.

Setting $q=1$ in \eqref{eq:functionalareaper} gives
$$\left(1-\frac{x (sy)^{2}}{1-sy} + \frac{xy s^2}{1-s} \right) F(s)
= \frac{xys}{1-ys}
+ \frac{xys^2}{1-s} F(1)
.$$ 

Let $p(m,n)$ denote the number of Fibonacci polyominoes whose boundary has $2m$ horizontal edges and $2n$ vertical edges, and define the corresponding generating function
$$P(x,y):=\sum_{m, n\geq 1} p(m,n)x^m y^n.$$
Thus $P(x,y)=F(x,y,1)$.

To compute $P(x,y)$, we apply the kernel method \cite{ban, banf, pro} to the above functional equation. This method consists of cancelling the \emph{kernel}
\[
K(s):=1-\frac{x(sy)^2}{1-sy}+\frac{xys^2}{1-s}
\]
by choosing $s$ to be a root  $s_0=s_0(x,y)$ of the equation $K(s)=0$. Substituting $s=s_0$ forces the left-hand side to vanish and thus  determines $F(1)$ from the right-hand side.

In our case, we take the root that is analytic at $x=0$ (the \emph{small} root), which is given by
\[
s_0=\frac{1+y-\sqrt{(1-y)(1-y-4xy)}}{2y\bigl(1+x(1-y)\bigr)}.
\]

\begin{theorem}\label{teorperimetro}
The generating function for nonempty Fibonacci polyominoes, counted by horizontal and vertical semiperimeter, is
\[P(x,y)=\frac{1 + x - y - 3 x y -(1+x)\sqrt{(1 - y) (1 - y - 4 x y)}}{2x}.\]
\end{theorem}

The first few terms of the expansion of $P(x,y)$ with respect to $y$ are
\begin{multline*}
 P(x,y)=   x y+\left(x^2+x\right) y^2+\left(2 x^3+3 x^2+x\right) y^3\\+\left(5 x^4+\bm{9 x^3}+5 x^2+x\right)
   y^4+\left(14 x^5+29 x^4+21 x^3+7 x^2+x\right) y^5+O\left(y^6\right).
\end{multline*}

The bold coefficients in the above expansion correspond to the Fibonacci polyominoes shown in Figure~\ref{fig2}.

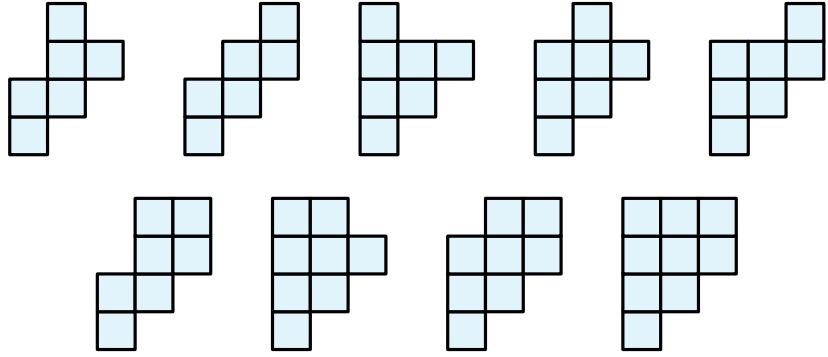
\begin{figure}[ht!]
\centering
\begin{tikzpicture}[scale=0.5, line cap=round, line join=round]
\tikzset{cell/.style={draw=black, very thick, fill=cyan!10!white}}
\foreach \x/\y in {0/0,0/1,1/1,1/2,1/3,2/2}{
  \draw[cell] (\x,\y) rectangle ++(1,1);
}
\end{tikzpicture}
\hspace{0.5cm}
\begin{tikzpicture}[scale=0.5, line cap=round, line join=round]
\tikzset{cell/.style={draw=black, very thick, fill=cyan!10!white}}
\foreach \x/\y in {0/0,0/1,1/1,1/2,2/2,2/3}{
  \draw[cell] (\x,\y) rectangle ++(1,1);
}
\end{tikzpicture}
\hspace{0.5cm}
\begin{tikzpicture}[scale=0.5, line cap=round, line join=round]
\tikzset{cell/.style={draw=black, very thick, fill=cyan!10!white}}
\foreach \x/\y in {0/0,0/1,0/2,0/3,1/1,1/2,2/2}{
  \draw[cell] (\x,\y) rectangle ++(1,1);
}
\end{tikzpicture}
\hspace{0.5cm}
\begin{tikzpicture}[scale=0.5, line cap=round, line join=round]
\tikzset{cell/.style={draw=black, very thick, fill=cyan!10!white}}
\foreach \x/\y in {0/0,0/1,0/2,1/1,1/2,1/3,2/2}{
  \draw[cell] (\x,\y) rectangle ++(1,1);
}
\end{tikzpicture}
\hspace{0.5cm}
\begin{tikzpicture}[scale=0.5, line cap=round, line join=round]
\tikzset{cell/.style={draw=black, very thick, fill=cyan!10!white}}
\foreach \x/\y in {0/0,0/1,0/2,1/1,1/2,2/2,2/3}{
  \draw[cell] (\x,\y) rectangle ++(1,1);
}
\end{tikzpicture}
\vspace{0.5cm}

\begin{tikzpicture}[scale=0.5, line cap=round, line join=round]
\tikzset{cell/.style={draw=black, very thick, fill=cyan!10!white}}
\foreach \x/\y in {0/0,0/1,1/1,1/2,1/3, 2/2,2/3}{
  \draw[cell] (\x,\y) rectangle ++(1,1);
}
\end{tikzpicture}
\hspace{0.5cm}
\begin{tikzpicture}[scale=0.5, line cap=round, line join=round]
\tikzset{cell/.style={draw=black, very thick, fill=cyan!10!white}}
\foreach \x/\y in {0/0,0/1,0/2,0/3,1/1,1/2,1/3,2/2}{
  \draw[cell] (\x,\y) rectangle ++(1,1);
}
\end{tikzpicture}
\hspace{0.5cm}
\begin{tikzpicture}[scale=0.5, line cap=round, line join=round]
\tikzset{cell/.style={draw=black, very thick, fill=cyan!10!white}}
\foreach \x/\y in {0/0,0/1,0/2,1/1,1/2,1/3,2/2,2/3}{
  \draw[cell] (\x,\y) rectangle ++(1,1);
}
\end{tikzpicture}
\hspace{0.5cm}
\begin{tikzpicture}[scale=0.5, line cap=round, line join=round]
\tikzset{cell/.style={draw=black, very thick, fill=cyan!10!white}}
\foreach \x/\y in {0/0,0/1,0/2,0/3,1/1,1/2,1/3,2/2,2/3}{
  \draw[cell] (\x,\y) rectangle ++(1,1);
}
\end{tikzpicture}
\caption{The nine Fibonacci polyominoes with horizontal semiperimeter $3$ and vertical semiperimeter $4$.}\label{fig2}
\end{figure}

Recall that the $n$th Catalan number is $C_n=\frac{1}{n+1}\binom{2n}{n}$, and that its generating function is  
$$C(z)=\sum_{n\geq 0}C_n z^n=\frac{1-\sqrt{1-4z}}{2z}.$$
 For numerous combinatorial objects enumerated by the Catalan numbers, we refer the reader to the book \cite{stan}. 

The following result gives a closed form for the coefficient $[x^n]P(x,y)$ in terms of Catalan numbers. 

\begin{theorem}
 For $n\geq 1$, 
 \[[x^n]P(x,y)=\frac{y^n}{(1-y)^n}\left(yC_n + (1-y)C_{n-1}\right).\]
\end{theorem}
\begin{proof}
We rewrite the square root in Theorem~\ref{teorperimetro} in terms of the generating function of the Catalan numbers. Set $z=xy/(1-y)$. Then
\[
\sqrt{(1-y)(1-y-4xy)}
= (1-y)\sqrt{1-\frac{4xy}{1-y}}
= (1-y)\sqrt{1-4z}.
\]
Since $\sqrt{1-4z}=1-2zC(z)$, we obtain
\[
\sqrt{(1-y)(1-y-4xy)}=(1-y)-2xy C(z).
\]
Substituting this expression into Theorem~\ref{teorperimetro} and simplifying gives
\[
P(x,y)=y\bigl((1+x)C(z)-1\bigr),
 \qquad z=\frac{xy}{1-y}.
\]
Expanding the above expression, we obtain 
\begin{align}\label{ExPCat}
P(x,y)
= \sum_{k\geq 1} C_k\, \frac{x^k y^{k+1}}{(1-y)^k}
+
\sum_{k\geq 0} C_k\, \frac{x^{k+1} y^{k+1}}{(1-y)^k}.
\end{align}
By comparing the coefficient of $x^n$ in $P(x,y)$, we immediately obtain the desired result.
\end{proof}

From \eqref{ExPCat} and the identity
\[
\frac{1}{(1-y)^k}
=
\sum_{m\geq 0}\binom{m+k-1}{m}y^m,
\]
we can similarly extract the coefficient of $y^n$.

\begin{theorem}
 For $n\geq 2$, we have 
 \[[y^n]P(x,y)=(1+x)x\sum_{k=0}^{n-2}C_{k+1}\binom{n-2}{k}x^k,\]
 and $[y]P(x,y)=x$.  
\end{theorem}

From the previous theorem we obtain the following expression for the coefficients of $P(x,y)$ in terms of Catalan numbers. 

\begin{theorem}\label{thm:perimeter-coefficients}
For $m\geq 1$ and $n\geq 2$,
\[
p(m,n)
=
[x^my^n]P(x,y)
=
C_m\binom{n-2}{m-1}
+
C_{m-1}\binom{n-2}{m-2}.
\]
In particular, $p(n,n)=C_{n-1}$ for every $n\geq 2$.
\end{theorem}

Although Theorem~\ref{thm:perimeter-coefficients} follows directly from
the generating function, its two summands suggest that the family admits
a more structural decomposition. The diagonal case already gives a first
indication of this phenomenon: if $m=n$, then a Fibonacci polyomino has
exactly $n$ columns and $n$ rows, and its upper boundary determines a
Dyck path of semilength $n-1$. Conversely, every Dyck path of semilength
$n-1$ determines a unique Fibonacci polyomino in $\F_{n,n}$. Hence $p(n,n)=C_{n-1}$.

For instance, when $n=4$, we obtain $p(4,4)=C_3=5$, see Figure~\ref{ex2}. In the next section, we extend this idea and give
a bijective explanation of the full formula in
Theorem~\ref{thm:perimeter-coefficients} in terms of labeled Dyck paths.

\begin{figure}[ht!]
\centering
\begin{tikzpicture}[scale=0.5, line cap=round, line join=round]
\tikzset{
  cell/.style={draw=black, very thick, fill=cyan!10!white},
  prof/.style={red, ultra thick}
}

%------------------ 1 ------------------
\begin{scope}[xshift=0cm,yshift=0cm]
\foreach \x/\y in {0/0,0/1,1/1,1/2,2/2,2/3,3/3}{
  \draw[cell] (\x,\y) rectangle ++(1,1);
}
\draw[prof] (0,1)--(0,2)--(1,2)--(1,3)--(2,3)--(2,4)--(3,4);
\end{scope}

%------------------ 2 ------------------
\begin{scope}[xshift=6.6cm,yshift=0cm]
\foreach \x/\y in {0/0,0/1,0/2,1/1,1/2,2/2,2/3,3/3}{
  \draw[cell] (\x,\y) rectangle ++(1,1);
}
\draw[prof] (0,1)--(0,3)--(1,3)--(2,3)--(2,4)--(3,4);
\end{scope}

%------------------ 3 ------------------
\begin{scope}[xshift=13.2cm,yshift=0cm]
\foreach \x/\y in {0/0,0/1,1/1,1/2,1/3,2/2,2/3,3/3}{
  \draw[cell] (\x,\y) rectangle ++(1,1);
}
\draw[prof] (0,1)--(0,2)--(1,2)--(1,3)--(1,4)--(2,4)--(3,4);
\end{scope}

%------------------ 4 ------------------
\begin{scope}[xshift=19.8cm,yshift=0cm]
\foreach \x/\y in {0/0,0/1,0/2,1/1,1/2,1/3,2/2,2/3,3/3}{
  \draw[cell] (\x,\y) rectangle ++(1,1);
}
\draw[prof] (0,1)--(0,3)--(1,3)--(1,4)--(2,4)--(3,4);
\end{scope}

%------------------ 5 ------------------
\begin{scope}[xshift=26.4cm,yshift=0cm]
\foreach \x/\y in {0/0,0/1,0/2,0/3,1/1,1/2,1/3,2/2,2/3,3/3}{
  \draw[cell] (\x,\y) rectangle ++(1,1);
}
\draw[prof] (0,1)--(0,4)--(1,4)--(2,4)--(3,4);
\end{scope}

\end{tikzpicture}
\caption{The five Fibonacci polyominoes in $\F_{4,4}$. In each case,
the upper boundary, shown in red, determines a Dyck path of semilength
$3$.}
\label{ex2}
\end{figure}
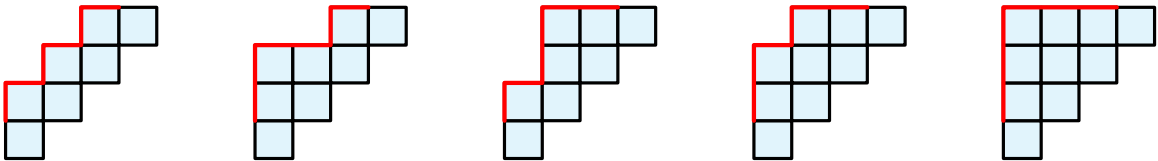

\section{A Dyck-Path Bijection}
\label{sec:Dyck-bijection}

In this section, we give a combinatorial explanation of Theorem~\ref{thm:perimeter-coefficients}. We establish a bijection between Fibonacci polyominoes with prescribed horizontal and vertical semiperimeters and Dyck paths whose up steps are labeled by nonnegative integers satisfying  certain conditions. The bijection will also yield refinements involving several natural statistics.  To achieve this objective, we divide the polyominoes into two classes: those whose last column consists of at least two  cells and those whose last column contains a single cell. Lemma~\ref{lem:decorated-Dyck} provides a combinatorial interpretation for the first class and allows us to derive one for the second class.

Let $P$ be a Fibonacci polyomino with $m$ columns, and write $h(P)=(h_1,h_2,\ldots,h_m)$ for its sequence of column heights. Its horizontal semiperimeter is
$m$, while its vertical semiperimeter is
\begin{equation}\label{eq:yper-heights}
\yper(P)
=
m-1+h_m
+
\sum_{i=1}^{m-1}
\max\{h_i-h_{i+1}-1,0\}.
\end{equation}

Indeed, the lower staircase contributes $m-1$ vertical edges, the
right boundary of the last column contributes $h_m$, and an additional
vertical contribution occurs along the upper boundary whenever
$h_i>h_{i+1}+1$, in which case its size is $h_i-h_{i+1}-1$.

We shall first consider Fibonacci polyominoes whose last column has
height at least $2$. For such a polyomino, it is convenient to reverse
the sequence of column heights and subtract $2$ from every entry,
obtaining $(h_m-2,h_{m-1}-2,\ldots,h_1-2)$, 
a sequence of nonnegative integers. This representation will be used
below to connect Fibonacci polyominoes with labeled Dyck paths.

We now introduce the sequence statistics needed for this correspondence.   For an integer $n$, write $(n)_+:=\max\{n,0\}$. For a sequence $\mathbf a=(a_1,\ldots,a_\ell)$ of nonnegative integers,
define 
\begin{align*}
    \operatorname{exc}(\mathbf a) &= a_1+\sum_{j=2}^{\ell}(a_j-a_{j-1}-1)_+,\\
    \asc(\mathbf{a}) &= \bigl|\{j:1\leq j<\ell,\ a_j<a_{j+1}\}\bigr|,\\
    \lev(\mathbf{a}) &= \bigl|\{j:1\leq j<\ell,\ a_j=a_{j+1}\}\bigr|.
\end{align*}  An entry $a_i$ is a \emph{strict right-to-left minimum} if $a_i<a_j$ for all $j>i$.  Let $\rlmin(\mathbf a)$ denote the number of strict right-to-left minima of $\mathbf a$. 

For $\ell\geq 1$ and $r\geq 0$, let $\mathcal{A}_{\ell,r}$ be the set
of nonnegative integer sequences $\mathbf{a}=(a_1,\ldots,a_\ell)$ such that $\operatorname{exc}(\mathbf{a})=r$.

Since our objective is to prove that $\mathcal{A}_{\ell,r}$ is in one-to-one correspondence with a class of labeled Dyck paths, we first provide some background on labeled Dyck paths.

A \emph{Dyck path} of semilength $\ell$ is a lattice path from $(0,0)$ to $(2\ell,0)$ with steps $U=(1,1)$ and $D=(1,-1)$ that never goes below the horizontal axis. Let $\mathcal{D}_\ell$ denote the set of Dyck paths of semilength $\ell$, and let $\epsilon$ denote the empty Dyck path, that is, the unique Dyck path of semilength $0$. A \emph{peak} is an occurrence of two consecutive steps $UD$. A \emph{return} is a down step whose endpoint lies on the horizontal axis. Thus the final down step of every nonempty Dyck path is a return. For a Dyck path $\pi$, let $\pk(\pi)$ and $\ret(\pi)$ denote its number of peaks and returns, respectively. A nonempty Dyck path is \emph{primitive} if it touches the horizontal axis only at its initial and terminal vertices. Equivalently, $\pi$ is primitive if $\ret(\pi)=1$. 

A \emph{labeled Dyck path} is a Dyck path in which each up step is assigned a nonnegative integer label. We represent a labeled Dyck path by a pair
$(\pi,\lambda)$, where $\pi$ is the underlying Dyck path and $\lambda=(\lambda_1,\ldots,\lambda_\ell)$ is the sequence of labels of its up steps, read from left to right. If $\lambda$ is such a labeling, write
\[
|\lambda|
=
\sum_{e\in\mathcal{U}(\pi)}\lambda(e)=\lambda_1+\cdots +\lambda_\ell,
\] 
where $\mathcal{U}(\pi)$ denotes the set of up steps of $\pi$. 
For $\ell\geq 0$ and $r\geq 0$, let $\mathcal{D}_{\ell,r}$ denote the set of
labeled Dyck paths $(\pi,\lambda)$ of semilength $\ell$ whose up-step labels
are nonnegative integers with total label $|\lambda|=r$. For example,
the labeled Dyck path in Figure~\ref{Fig3} belongs to
$\mathcal{D}_{9,27}$.

Let $(\pi,\lambda)$ be a nonempty labeled Dyck path, with first-return decomposition  $\pi=U\pi_LD\pi_R$, and let $\lambda_L$ and $\lambda_R$ be the labelings induced on $\pi_L$ and $\pi_R$, respectively. Define 
\[ \delta(\pi_L,\lambda_L):= \begin{cases} 1, &\text{if $\pi_L$ is nonempty and primitive, and the initial up}\\ &\text{step of $\pi_L$ has label $0$},\\ 0, &\text{otherwise}. \end{cases} \]

Then we recursively define the statistic $\lv$ by
\[
\lv(\pi,\lambda)
=
\lv(\pi_L,\lambda_L)
+
\lv(\pi_R,\lambda_R)
+
\delta(\pi_L,\lambda_L).
\] 
Intuitively, $\lv(\pi,\lambda)$ counts the occurrences, throughout the
recursive first-return decomposition, in which the left subpath is
nonempty and primitive and its initial up step has label $0$. Equivalently, $\lv(\pi,\lambda)$ is the number of factors
$UU\sigma DD$ of $\pi$, where $\sigma$ is a possibly empty Dyck path
and the second up step is labeled $0$. For
example, Figure~\ref{Fig3} shows a labeled Dyck path $\pi$ of semilength $9$ whose up-step labels, read from left to right, are   $\lambda=(3,5,2,7,4,0,6,0,0)$.  Applying the recursive definition, one verifies that $\lv(\pi,\lambda)=3$.

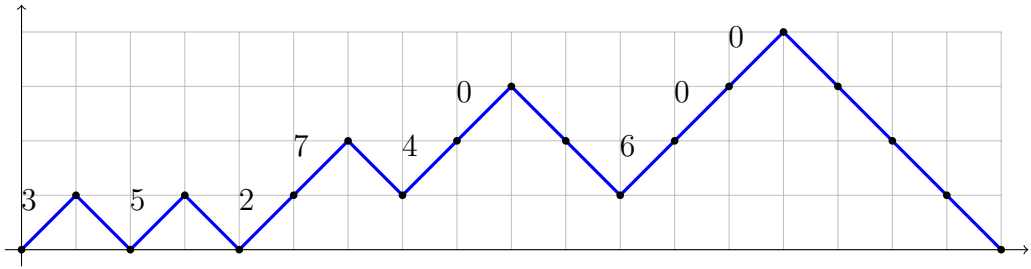
\begin{figure}[ht!]
\centering
\begin{tikzpicture}[scale=0.72, line cap=round, line join=round]
  % Grid
  \draw[step=1cm, very thin, gray!50] (0,0) grid (18,4);

  % Axes
  \draw[->] (-0.3,0) -- (18.5,0);
  \draw[->] (0,-0.3) -- (0,4.5);

  % Dyck path of semilength 9
  \draw[very thick, blue]
    (0,0)
    -- node[midway, above left, text=black] {$3$} (1,1)
    -- (2,0)
    -- node[midway, above left, text=black] {$5$} (3,1)
    -- (4,0)
    -- node[midway, above left, text=black] {$2$} (5,1)
    -- node[midway, above left, text=black] {$7$} (6,2)
    -- (7,1)
    -- node[midway, above left, text=black] {$4$} (8,2)
    -- node[midway, above left, text=black] {$0$} (9,3)
    -- (10,2)
    -- (11,1)
    -- node[midway, above left, text=black] {$6$} (12,2)
    -- node[midway, above left, text=black] {$0$} (13,3)
    -- node[midway, above left, text=black] {$0$} (14,4)
    -- (15,3)
    -- (16,2)
    -- (17,1)
    -- (18,0);

  % Vertices of the path
  \foreach \x/\y in {
    0/0,1/1,2/0,3/1,4/0,5/1,6/2,7/1,8/2,9/3,
    10/2,11/1,12/2,13/3,14/4,15/3,16/2,17/1,18/0
  }
    \fill (\x,\y) circle (2pt);

\end{tikzpicture}
\caption{A labeled Dyck path of semilength $9$ with
$\lv(\pi,\lambda)=3$.}
\label{Fig3}
\end{figure}

\begin{lemma}\label{lem:decorated-Dyck}
There is a bijection $\Phi$ between the set $\mathcal{A}_{\ell,r}$ of nonnegative sequences $\mathbf a=(a_1,\ldots, a_\ell)$ with  $\operatorname{exc}(\mathbf a)=r$, and the set $\mathcal{D}_{\ell,r}$ of labeled  Dyck paths of
semilength $\ell$ with total label sum equal to $r$.  Under this bijection, if $\mathbf{a}\in\mathcal{A}_{\ell,r}$ and $(\pi,\lambda)=\Phi(\mathbf a)$, then
\[
\asc(\mathbf{a})=\pk(\pi)-1,
\qquad
\lev(\mathbf{a})=\lv(\pi,\lambda),
\qquad
\rlmin(\mathbf{a})=\ret(\pi).
\]
Moreover, the label of the initial up step of $\pi$ is $\min\{a_1,\ldots,a_\ell\}$.  Consequently, 
\[
\lvert\mathcal{A}_{\ell,r}\rvert
=
C_\ell\binom{r+\ell-1}{\ell-1}.
\]
\end{lemma}
\begin{proof}
  We construct recursively the bijection  $\Phi:\mathcal{A}_{\ell,r}\longrightarrow\mathcal{D}_{\ell,r}$.   For convenience, the empty sequence
corresponds to the empty Dyck path $\epsilon$. Assuming $\ell>0$, let $\mathbf{a}=(a_1,\ldots,a_\ell)\in\mathcal{A}_{\ell,r}$, and set $\rho=\min\{a_1,\ldots,a_\ell\}$. Let $k$ be the rightmost index such that $a_k=\rho$. Define
\[
\mathbf{a}^{L}
=
(a_1-\rho,\ldots,a_{k-1}-\rho) \qquad \text{and} \qquad 
\mathbf{a}^{R}
=
(a_{k+1}-\rho-1,\ldots,a_\ell-\rho-1).
\]
Since $k$ is the rightmost occurrence of the minimum, every entry to its
right is at least $\rho+1$. Thus $\mathbf{a}^{L}$ and
$\mathbf{a}^{R}$ are sequences of nonnegative integers.

  Apply the construction recursively to obtain $\Phi(\mathbf{a}^{L})=(\pi_L,\lambda_L)$ and $\Phi(\mathbf{a}^{R})=(\pi_R,\lambda_R)$. We define
\[
\Phi(\mathbf{a})=(\pi,\lambda),
 \text{ where } 
\pi=U\pi_LD\pi_R \text{ and } 
\]
$\lambda$ is the sequence of nonnegative integers consisting of the juxtaposition of $\rho$, $\lambda_L$ and $\lambda_R$.
 The initial up step of
$\pi$ has label $\rho=\min\{a_1,\ldots,a_\ell\}$.

We next show that the construction is invertible. Let $(\pi,\lambda)$ be
a nonempty labeled Dyck path, and write its first-return decomposition as $\pi=U\pi_LD\pi_R$. Let $\rho$ be the label of the initial up step. Recursively recover
$\mathbf{a}^{L}$ and $\mathbf{a}^{R}$ from the labeled paths
$(\pi_L,\lambda_L)$ and $(\pi_R,\lambda_R)$, respectively, and set
\[
\mathbf{a}
=
\bigl(
\rho+\mathbf{a}^{L},
\ \rho,\
\rho+1+\mathbf{a}^{R}
\bigr),
\]
where addition to a sequence is understood componentwise. Every entry to the
left of the central $\rho$ is at least $\rho$, whereas every entry to its
right is at least $\rho+1$. Hence the central entry is the rightmost
occurrence of the minimum of $\mathbf{a}$. It follows that the two
constructions are inverse to each other.

We now verify the statistics in the statement. We do this recursively,
with all sequence statistics understood to be $0$ on the empty sequence.

First, from the decomposition of $\mathbf a$ one obtains
\[
\operatorname{exc}(\mathbf a)
=
\rho+\operatorname{exc}(\mathbf a^L)+\operatorname{exc}(\mathbf a^R).
\]
Indeed, adding $\rho$ to all entries of $\mathbf a^L$ contributes the extra initial term $\rho$
and leaves all subsequent excesses unchanged. The transition from $\mathbf a^L$ to the central
entry $\rho$ does not contribute to $\operatorname{exc}$, while the transition from
$\rho$ to the first entry on its right, when present, gives exactly the
first term in $\operatorname{exc}(\mathbf{a}^{R})$.  By induction,
$\operatorname{exc}(\mathbf{a}^{L})$ and
$\operatorname{exc}(\mathbf{a}^{R})$ are the sums of the labels of
$\pi_L$ and $\pi_R$, respectively. Since the initial up step has label
$\rho$, we have
\[
|\lambda|
=
\rho+|\lambda_L|+|\lambda_R|
=
\operatorname{exc}(\mathbf{a})
=
r.
\]
Thus $\Phi$ restricts to a bijection between $\mathcal{A}_{\ell,r}$ and
Dyck paths of semilength $\ell$ with total label~$r$.

We next consider ascents. No ascent occurs between the last entry of
$\mathbf{a}^{L}$ and the central entry $\rho$. On the other hand, if
$\mathbf{a}^{R}$ is nonempty, then there is always an ascent from $\rho$
to the first entry on its right. Therefore,
\[
\asc(\mathbf{a})
=
\asc(\mathbf{a}^{L})
+
\asc(\mathbf{a}^{R})
+
\mathbf{1}_{\{\mathbf{a}^{R}\neq\emptyset\}}.
\]
Here $\mathbf{1}_{\{A\}}$ denotes the indicator of the condition $A$, equal to $1$ if $A$ holds and $0$ otherwise. 

For the corresponding first-return decomposition
$\pi=U\pi_LD\pi_R$, the initial $U$ and its matching $D$ form a peak
precisely when $\pi_L$ is empty. Therefore,
\[
\pk(\pi)
=
\pk(\pi_L)
+
\pk(\pi_R)
+
\mathbf{1}_{\{\pi_L=\epsilon\}}.
\]
Using the induction hypothesis on $\pi_L$ and $\pi_R$, we obtain
\[
\pk(\pi_L)
=
\asc(\mathbf a^L)
+
\mathbf{1}_{\{\mathbf a^L\neq\emptyset\}},
\qquad
\pk(\pi_R)
=
\asc(\mathbf a^R)
+
\mathbf{1}_{\{\mathbf a^R\neq\emptyset\}}.
\]
Since $\pi_L=\epsilon$ if and only if $\mathbf a^L=\emptyset$, it follows that
\[
\pk(\pi)
=
\asc(\mathbf a^L)
+
\asc(\mathbf a^R)
+
\mathbf{1}_{\{\mathbf a^R\neq\emptyset\}}
+1
=
\asc(\mathbf a)+1.
\]
Equivalently, $\asc(\mathbf{a})=\pk(\pi)-1$.

We now consider strict right-to-left minima. The central entry $\rho$ is
a strict right-to-left minimum because every entry to its right is
strictly larger than $\rho$. No entry to its left can be a strict
right-to-left minimum, since $\rho$ occurs later in the sequence. The strict
right-to-left minima to the right of $\rho$ are precisely those of
$\mathbf{a}^{R}$. Hence $\rlmin(\mathbf{a})
=
1+\rlmin(\mathbf{a}^{R})$. On the path side, the down step following $\pi_L$ is the first return of
$\pi$, and all subsequent returns are exactly the returns of $\pi_R$.
Thus $
\ret(\pi)
=
1+\ret(\pi_R)$. The induction hypothesis now gives $\rlmin(\mathbf{a})=\ret(\pi)$.

It remains to consider levels. A level can occur inside
$\mathbf{a}^{L}$ or $\mathbf{a}^{R}$. There is never a level between
the central entry $\rho$ and the first entry to its right, since the
latter is at least $\rho+1$. Thus an additional level occurs only
between the last entry on the left and $\rho$. This happens precisely
when $\mathbf{a}^{L}$ is nonempty and its last entry is $0$. Therefore,
\[
\lev(\mathbf{a})
=
\lev(\mathbf{a}^{L})
+
\lev(\mathbf{a}^{R})
+
\begin{cases}
1, & \text{if $\mathbf{a}^{L}$ is nonempty and ends in $0$};\\
0, & \text{otherwise}.
\end{cases}
\]

We claim that $\mathbf{a}^{L}$ is nonempty and ends in $0$ if and only if
$\pi_L$ is nonempty and primitive and its initial up step has label $0$.
Indeed, suppose first that $\mathbf{a}^{L}$ ends in $0$. Since all its
entries are nonnegative, its minimum is $0$, so the initial up step of
$\pi_L$ has label $0$. Moreover, the final $0$ is the only strict
right-to-left minimum of $\mathbf{a}^{L}$. By the result just proved,
$\ret(\pi_L)=1$, and hence $\pi_L$ is primitive. Conversely, the same argument in reverse applies: if $\pi_L$ is primitive
and its initial up step has label $0$, then $\mathbf{a}^{L}$ has minimum
$0$ and exactly one strict right-to-left minimum, which forces its final
entry to be $0$.

It follows that the additional term in the recurrence for
$\lev(\mathbf{a})$ is precisely
$\delta(\pi_L,\lambda_L)$. Using the induction hypothesis gives
\[
\lev(\mathbf{a})
=
\lv(\pi_L,\lambda_L)
+
\lv(\pi_R,\lambda_R)
+
\delta(\pi_L,\lambda_L)
=
\lv(\pi,\lambda).
\]

Finally, there are $C_\ell$ Dyck paths of semilength $\ell$. Once such a
path is fixed, assigning nonnegative labels of total sum $r$ to its
$\ell$ up steps is equivalent to a weak composition of $r$ into $\ell$
parts. There are $\binom{r+\ell-1}{\ell-1}$ such labelings. Consequently,
\[
\lvert\mathcal{A}_{\ell,r}\rvert
=
C_\ell\binom{r+\ell-1}{\ell-1}. \qedhere
\]
\end{proof}

For a concrete illustration of the bijection in Lemma~3.1, consider
$\mathbf a=(4,2,3,1,3,1,3,3)$. Figure~\ref{fig:decomposition-tree-dyck} shows its recursive decomposition
and the corresponding labeled Dyck path. At each node, the highlighted
entry is the rightmost occurrence of the minimum $\rho$. To read the tree,
we write an up step when visiting a node, recursively traverse its left
subtree, write a down step, and then traverse its right subtree. Thus each
node contributes exactly one pair $U,D$, with $\rho$ giving the label of
the corresponding up step. The numbering $\#1,\ldots,\#8$ records the
preorder traversal, which is precisely the order in which the up steps are
encountered.

Thus $\pi=UUUUDDUDDUDDUUDD$ and  $\lambda=(1,0,1,2,0,1,1,0)$. The statistics in Lemma~3.1 can be checked directly. Indeed,
\[
\exc(\mathbf a)
=
4+(3-2-1)_+ +(3-1-1)_+ +(3-1-1)_+
=
6,
\]
while $|\lambda|=1+0+1+2+0+1+1+0=6$. Moreover,
\[
\asc(\mathbf a)=3,\qquad
\lev(\mathbf a)=1,\qquad
\rlmin(\mathbf a)=2,
\]
whereas the corresponding Dyck path satisfies
\[
\pk(\pi)=4,\qquad
\lv(\pi,\lambda)=1,\qquad
\ret(\pi)=2.
\]
Notice also that the label of the initial up step is $1$, which is
precisely the minimum entry of $\mathbf a$.

\begin{figure}[ht!]
\centering

%------------------------------------------------
% Recursive decomposition tree
%------------------------------------------------
\begin{minipage}{0.48\textwidth}
\centering
\begin{forest}
for tree={
    draw,
    rounded corners,
    align=center,
    parent anchor=south,
    child anchor=north,
    l sep=10mm,
    s sep=7mm,
    inner sep=2pt,
    font=\scriptsize
}
[{$\#1$\\
$(4,2,3,1,3,\highlightmin{1},3,3)$\\
$\rho=1$}
    [{$\#2$\\
    $(3,1,2,\highlightmin{0},2)$\\
    $\rho=0$}
        [{$\#3$\\
        $(3,\highlightmin{1},2)$\\
        $\rho=1$}
            [{$\#4$\\
            $(\highlightmin{2})$\\
            $\rho=2$}]
            [{$\#5$\\
            $(\highlightmin{0})$\\
            $\rho=0$}]
        ]
        [{$\#6$\\
        $(\highlightmin{1})$\\
        $\rho=1$}]
    ]
    [{$\#7$\\
    $(1,\highlightmin{1})$\\
    $\rho=1$}
        [{$\#8$\\
        $(\highlightmin{0})$\\
        $\rho=0$}]
    ]
]
\end{forest}
\end{minipage}
\hfill
%------------------------------------------------
% Corresponding labeled Dyck path
%------------------------------------------------
\begin{minipage}{0.49\textwidth}
\centering

\begin{tikzpicture}[scale=0.43]

    % Dyck path
    \draw[very thick]
    (0,0)
    --(1,1)
    --(2,2)
    --(3,3)
    --(4,4)
    --(5,3)
    --(6,2)
    --(7,3)
    --(8,2)
    --(9,1)
    --(10,2)
    --(11,1)
    --(12,0)
    --(13,1)
    --(14,2)
    --(15,1)
    --(16,0);

    % Horizontal axis
    \draw[gray] (0,0)--(16,0);

    % Vertices of the Dyck path
    \foreach \x/\y in {
        0/0,
        1/1,
        2/2,
        3/3,
        4/4,
        5/3,
        6/2,
        7/3,
        8/2,
        9/1,
        10/2,
        11/1,
        12/0,
        13/1,
        14/2,
        15/1,
        16/0
    }
    {
        \fill (\x,\y) circle (2pt);
    }

    % Labels of the up steps
    \node[above] at (0.5,0.5)  {\scriptsize $1$};
    \node[above] at (1.5,1.5)  {\scriptsize $0$};
    \node[above] at (2.5,2.5)  {\scriptsize $1$};
    \node[above] at (3.5,3.5)  {\scriptsize $2$};
    \node[above] at (6.5,2.5)  {\scriptsize $0$};
    \node[above] at (9.5,1.5)  {\scriptsize $1$};
    \node[above] at (12.5,0.5) {\scriptsize $1$};
    \node[above] at (13.5,1.5) {\scriptsize $0$};

    % Preorder indices
    \node[below left] at (1.5,0.5)  {\tiny $\#1$};
    \node[below left] at (1.5+1,1.5)  {\tiny $\#2$};
    \node[below left] at (2.5+1,2.5)  {\tiny $\#3$};
    \node[below left] at (3.5+1,3.5)  {\tiny $\#4$};
    \node[below left] at (6.5+1,2.5)  {\tiny $\#5$};
    \node[below left] at (9.5+1,1.5)  {\tiny $\#6$};
    \node[below left] at (12.5+1,0.5) {\tiny $\#7$};
    \node[below left] at (13.5+1,1.5) {\tiny $\#8$};

\end{tikzpicture}

\vspace{2mm}

{\scriptsize
$\pi=UUUUDDUDDUDDUUDD$
}

\vspace{1mm}

{\scriptsize
$\lambda=(1,0,1,2,0,1,1,0)$
}

\end{minipage}

\caption{Recursive decomposition tree for
$\mathbf a=(4,2,3,1,3,1,3,3)$ and the corresponding labeled Dyck path.}
\label{fig:decomposition-tree-dyck}
\end{figure}
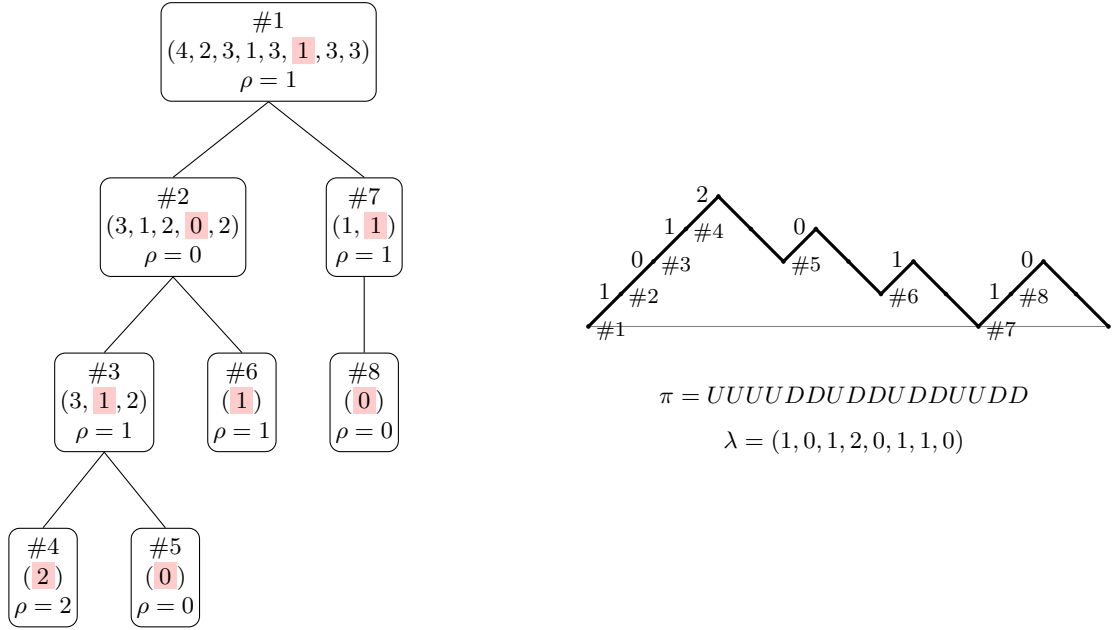

The recursive construction in Lemma~\ref{lem:decorated-Dyck} also has a
natural binary-tree interpretation. Using the same decomposition
$\mathbf a=(\rho+\mathbf a^L,\rho,\rho+1+\mathbf a^R)$, define
$\mathcal T(\mathbf a)$ recursively as the rooted plane binary tree whose
root has label $\rho$, with left and right subtrees
$\mathcal T(\mathbf a^L)$ and $\mathcal T(\mathbf a^R)$, respectively.
The empty sequence corresponds to the empty tree.

The inorder traversal of $\mathcal T(\mathbf a)$ recovers the original
order of the entries of $\mathbf a$: its $i$th node arises from $a_i$,
although its label need not equal $a_i$. Under the standard bijection
between rooted plane binary trees and Dyck paths, $\mathcal T(\mathbf a)$
corresponds precisely to the Dyck path $\pi$ in
Lemma~\ref{lem:decorated-Dyck}, while the node labels read in preorder
give the labels of the up steps of $\pi$. Figure~\ref{fig:tree}
illustrates the construction for
$\mathbf a=(4,2,3,1,3,1,3,3)$.

\begin{figure}[ht!]
\centering
\begin{tikzpicture}[
  scale=0.8,
  transform shape,
  level distance=13mm,
  level 1/.style={sibling distance=42mm},
  level 2/.style={sibling distance=22mm},
  level 3/.style={sibling distance=12mm},
  every node/.style={circle,draw,minimum size=7mm,inner sep=1pt},
  edge from parent/.style={draw,-}
]
\node {$1$}
  child {
    node {$0$}
    child {
      node {$1$}
      child { node {$2$} }
      child { node {$0$} }
    }
    child {
      node {$1$}
    }
  }
  child {
    node {$1$}
    child {
      node {$0$}
    }
  };
\end{tikzpicture}
\caption{The labeled binary tree associated with the sequence
$\mathbf a=(4,2,3,1,3,1,3,3)$.}
\label{fig:tree}
\end{figure}
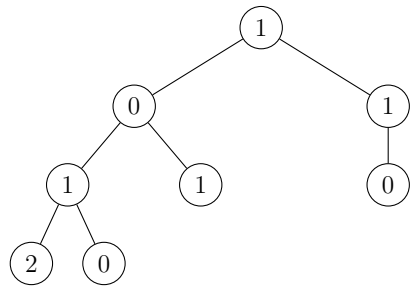

There is also a direct reconstruction of the sequence from the labeled tree. Let $v_1,\ldots,v_\ell$ be the nodes of $T(a)$ read in inorder.
For a node $v$, let $\operatorname{rdep}(v)$ denote the number of right
edges on the path from the root to $v$, and set
\[
\Lambda_T(v)=\sum_{w\in P_T(v)}\lambda(w),
\]
where $P(v)$ is the set of nodes on this path, including its endpoints.  We omit the subscript $T$ when there is no ambiguity. For example, in Figure~\ref{fig:tree}, the nodes $v_1$ and $v_3$
correspond to $a_1=4$ and $a_3=3$, respectively. Their root paths have
labels $1,0,1,2$ and $1,0,1,0$, and contain respectively zero and one
right edges. Hence
\[
\begin{aligned}
a_1&=\operatorname{rdep}(v_1)+\Lambda(v_1)
    =0+(1+0+1+2)=4,\\
a_3&=\operatorname{rdep}(v_3)+\Lambda(v_3)
    =1+(1+0+1+0)=3.
\end{aligned}
\]
In general, $a_i=\operatorname{rdep}(v_i)+\Lambda(v_i)$, for $1\leq i\leq\ell$. Thus the labeled tree determines the sequence uniquely.

The statistics of Lemma~\ref{lem:decorated-Dyck} also have simple tree
interpretations:
\begin{align*}
\operatorname{exc}(\mathbf a)
&=\sum_{v\in T}\lambda(v),\\
\operatorname{asc}(\mathbf a)
&=\bigl|\{v\in T:v\text{ has no left child}\}\bigr|-1,\\
\operatorname{rlmin}(\mathbf a)
&=\bigl|\operatorname{Rspine}(T)\bigr|,
\end{align*}
where $\operatorname{Rspine}(T)$ denotes the right spine of $T$.
Moreover, $\operatorname{lev}(\mathbf a)$ counts the nodes $v$ whose
left child $w$ has no right child and satisfies $\lambda(w)=0$.
For the tree in Figure~\ref{fig:tree}, the right spine consists of the
root and its right child, so $\operatorname{rlmin}(\mathbf a)=2$.

We now apply Lemma~\ref{lem:decorated-Dyck} to Fibonacci polyominoes.
The following result gives a combinatorial interpretation of the two
terms in Theorem~\ref{thm:perimeter-coefficients}.

\begin{theorem}\label{thm:Dyck-polyomino-bijection}
Let $m\geq 1$ and $n\geq 2$. Then
\[
p(m,n)
=
[x^my^n]P(x,y)
=
C_m\binom{n-2}{m-1}
+
C_{m-1}\binom{n-2}{m-2}.
\]
Moreover, the first term counts the Fibonacci polyominoes whose last
column has height at least $2$, while the second term counts those whose
last column has height $1$. In particular,
$p(n,n)=C_{n-1}$, 
for every $n\geq 2$.
\end{theorem}
\begin{proof}
Let $P$ be a Fibonacci polyomino with $m$ columns, and let
$(h_1,h_2,\ldots,h_m)$ be its sequence of column heights. Recall that $h_i\geq 2$ for $1\leq i\leq m-1$, while $h_m\geq 1$. Moreover, the vertical semiperimeter satisfies
 (see Equation \ref{eq:yper-heights})
\[\yper(P)
=
m-1+h_m
+
\sum_{i=1}^{m-1}(h_i-h_{i+1}-1)_+.
\]
Since $P\in\F_{m,n}$, the left-hand side is $n$.

We distinguish two cases according to the height of the last column.

\medskip
\noindent
\emph{Case 1: $h_m\geq 2$.}
Define
\[
\mathbf a(P)
=
(h_m-2,h_{m-1}-2,\ldots,h_1-2).
\]
This is a sequence of $m$ nonnegative integers. By
\eqref{eq:yper-heights},
\begin{align*}
\operatorname{exc}(\mathbf a(P))
&=
h_m-2
+
\sum_{i=1}^{m-1}(h_i-h_{i+1}-1)_+=n-m-1.
\end{align*}
Thus $\mathbf a(P)\in\mathcal A_{m,n-m-1}$. Conversely, every sequence in $\mathcal A_{m,n-m-1}$ uniquely determines
the column heights of such a polyomino by reversing the sequence and adding
$2$ to each entry. Hence this is a bijection.

By Lemma~\ref{lem:decorated-Dyck}, these polyominoes are therefore in
bijection with Dyck paths of semilength $m$ whose up-step labels are
nonnegative integers with total label $n-m-1$. For each fixed Dyck path,
the number of such labelings is $\binom{(n-m-1)+m-1}{m-1}
=
\binom{n-2}{m-1}$. Since there are $C_m$ Dyck paths of semilength $m$, the number of
polyominoes in this case is $C_m\binom{n-2}{m-1}$.

\medskip
\noindent
\emph{Case 2: $h_m=1$.}
If $m=1$, then $P$ consists of a single cell, so $n=1$. Since we are
assuming $n\geq 2$, we may suppose that $m\geq 2$. Define
\[
\mathbf a(P)
=
(h_{m-1}-2,h_{m-2}-2,\ldots,h_1-2).
\]
Again this is a sequence of nonnegative integers. Since $h_m=1$ and
$h_{m-1}\geq 2$, equation~\eqref{eq:yper-heights} gives
\begin{align*}
n
&=
m
+
(h_{m-1}-2)
+
\sum_{i=1}^{m-2}(h_i-h_{i+1}-1)_+,
\end{align*}
and therefore $\operatorname{exc}(\mathbf a(P))
=
n-m$. Thus
$\mathbf a(P)\in\mathcal A_{m-1,n-m}$. As in the first case, this construction is reversible.

Lemma~\ref{lem:decorated-Dyck} now gives a bijection with Dyck paths of
semilength $m-1$ whose up-step labels have total sum $n-m$. For each
fixed path, the number of such labelings is $
\binom{(n-m)+(m-1)-1}{m-2}
=
\binom{n-2}{m-2}$. Since there are $C_{m-1}$ Dyck paths of semilength $m-1$, this case
contains $C_{m-1}\binom{n-2}{m-2}$ 
polyominoes. Adding the two disjoint cases gives the desired result. 
\end{proof}

For example, consider $(m,n)=(2,5)$. By
Theorem~\ref{thm:Dyck-polyomino-bijection},
\[
p(2,5)
=
C_2\binom{3}{1}
+
C_1\binom{3}{0}
=
6+1
=
7.
\]
Table~\ref{tab:mn25} lists the seven Fibonacci polyominoes in
$\F_{2,5}$, together with the associated sequence $\mathbf a(P)$,
Dyck path $\pi$, and label sequence $\lambda$. If $h_2\geq2$, then $\mathbf a(P)\in\mathcal A_{2,2}$.
Thus each of the $C_2=2$ Dyck paths of semilength $2$ occurs with
the three possible label sequences of total weight $2$, namely
$(2,0)$, $(1,1)$, and $(0,2)$. If $h_2=1$, then
$\mathbf a(P)\in\mathcal A_{1,3}$, giving the unique Dyck path of
semilength $1$ with label $3$.

\begin{table}[ht!]
\centering
\small
\tikzset{cell/.style={draw=black, very thick, fill=cyan!10!white}}
\renewcommand{\arraystretch}{1.4}
\setlength{\tabcolsep}{6pt}
\begin{tabular}{c c c c c}
\hline
$P$ & $(h_1,h_2)$ & $\mathbf a(P)$ & $\pi$ & $\lambda$\\
\hline
\multicolumn{5}{c}{$h_2\geq 2$}\\
\hline
\\[-1.2em]
\begin{tikzpicture}[scale=0.22, baseline=-0.1cm,
                    line cap=round, line join=round]
\foreach \x/\y in {
0/0,0/1,0/2,0/3,
1/1,1/2,1/3,1/4}{
  \draw[cell] (\x,\y) rectangle ++(1,1);
}
\end{tikzpicture}
&
$(4,4)$
&
$(2,2)$
&
\begin{tikzpicture}[scale=0.22, baseline=-0.1cm,
                    line cap=round, line join=round]
\draw[very thick]
(0,0)--(1,1)--(2,2)--(3,1)--(4,0);
\end{tikzpicture}
&
$(2,0)$
\\

\begin{tikzpicture}[scale=0.22, baseline=-0.1cm,
                    line cap=round, line join=round]
\foreach \x/\y in {
0/0,0/1,0/2,
1/1,1/2,1/3,1/4}{
  \draw[cell] (\x,\y) rectangle ++(1,1);
}
\end{tikzpicture}
&
$(3,4)$
&
$(2,1)$
&
\begin{tikzpicture}[scale=0.22, baseline=-0.1cm,
                    line cap=round, line join=round]
\draw[very thick]
(0,0)--(1,1)--(2,2)--(3,1)--(4,0);
\end{tikzpicture}
&
$(1,1)$
\\

\begin{tikzpicture}[scale=0.22, baseline=-0.1cm,
                    line cap=round, line join=round]
\foreach \x/\y in {
0/0,0/1,
1/1,1/2,1/3,1/4}{
  \draw[cell] (\x,\y) rectangle ++(1,1);
}
\end{tikzpicture}
&
$(2,4)$
&
$(2,0)$
&
\begin{tikzpicture}[scale=0.22, baseline=-0.1cm,
                    line cap=round, line join=round]
\draw[very thick]
(0,0)--(1,1)--(2,2)--(3,1)--(4,0);
\end{tikzpicture}
&
$(0,2)$
\\

\hline
\\[-1.2em]
\begin{tikzpicture}[scale=0.22, baseline=-0.1cm,
                    line cap=round, line join=round]
\foreach \x/\y in {
0/0,0/1,0/2,0/3,0/4,
1/1,1/2,1/3,1/4}{
  \draw[cell] (\x,\y) rectangle ++(1,1);
}
\end{tikzpicture}
&
$(5,4)$
&
$(2,3)$
&
\begin{tikzpicture}[scale=0.22, baseline=-0.1cm,
                    line cap=round, line join=round]
\draw[very thick]
(0,0)--(1,1)--(2,0)--(3,1)--(4,0);
\end{tikzpicture}
&
$(2,0)$
\\

\begin{tikzpicture}[scale=0.22, baseline=-0.1cm,
                    line cap=round, line join=round]
\foreach \x/\y in {
0/0,0/1,0/2,0/3,0/4,
1/1,1/2,1/3}{
  \draw[cell] (\x,\y) rectangle ++(1,1);
}
\end{tikzpicture}
&
$(5,3)$
&
$(1,3)$
&
\begin{tikzpicture}[scale=0.22, baseline=-0.1cm,
                    line cap=round, line join=round]
\draw[very thick]
(0,0)--(1,1)--(2,0)--(3,1)--(4,0);
\end{tikzpicture}
&
$(1,1)$
\\

\begin{tikzpicture}[scale=0.22, baseline=-0.1cm,
                    line cap=round, line join=round]
\foreach \x/\y in {
0/0,0/1,0/2,0/3,0/4,
1/1,1/2}{
  \draw[cell] (\x,\y) rectangle ++(1,1);
}
\end{tikzpicture}
&
$(5,2)$
&
$(0,3)$
&
\begin{tikzpicture}[scale=0.22, baseline=-0.1cm,
                    line cap=round, line join=round]
\draw[very thick]
(0,0)--(1,1)--(2,0)--(3,1)--(4,0);
\end{tikzpicture}
&
$(0,2)$
\\

\hline
\multicolumn{5}{c}{$h_2=1$}\\
\hline
\\[-1.2em]
\begin{tikzpicture}[scale=0.22, baseline=-0.1cm,
                    line cap=round, line join=round]
\foreach \x/\y in {
0/0,0/1,0/2,0/3,0/4,
1/1}{
  \draw[cell] (\x,\y) rectangle ++(1,1);
}
\end{tikzpicture}
&
$(5,1)$
&
$(3)$
&
\begin{tikzpicture}[scale=0.22, baseline=-0.1cm,
                    line cap=round, line join=round]
\draw[very thick]
(0,0)--(1,1)--(2,0);
\end{tikzpicture}
&
$(3)$
\\

\hline\\[-1.2em]
\end{tabular}
\caption{The bijection of
Theorem~\ref{thm:Dyck-polyomino-bijection} for $(m,n)=(2,5)$.}
\label{tab:mn25}
\end{table}

\section{Further Refinements}\label{sec:further-refinements}

In this section, we use the bijection of
Lemma~\ref{lem:decorated-Dyck} to refine the enumeration of Fibonacci
polyominoes simultaneously by descents, left-to-right minima, and
levels. The first refinement is governed by the joint distribution of
peaks and returns on Dyck paths, while the second leads naturally to
Motzkin numbers.

\subsection{The joint distribution of descents and left-to-right minima}
\label{subsec:descents-minima}

Let $P$ be a Fibonacci polyomino with column heights
$(h_1,\ldots,h_m)$. Define
\[
\des(P)
=
\bigl|\{i:1\leq i<m,\ h_i>h_{i+1}\}\bigr|.
\]
A column $i\geq2$ is a \emph{left-to-right minimum} if
\[
h_i<\min\{h_1,\ldots,h_{i-1}\},
\]
and the first column is counted by convention. We denote by $\lmin(P)$
the number of left-to-right minima of $P$.

The correspondence in Lemma~\ref{lem:decorated-Dyck} translates these
two statistics into peaks and returns of Dyck paths. We use the joint
distribution of these statistics given by Deutsch
\cite[Section~6.5, Eq.~(6.20)]{DeutschDyck}.

For $\ell\geq1$, let $R(\ell,k,j)$ denote the number of Dyck paths of
semilength $\ell$ with exactly $k$ peaks and $j$ returns. Then
\[
R(\ell,k,j)
=
\begin{cases}
\displaystyle
\frac{j}{\ell}
\binom{\ell}{k}
\binom{\ell-j-1}{k-j},
& 1\leq j\leq k\leq\ell,\quad j<\ell,\\[3mm]
1,
& k=j=\ell,\\
0,
& \text{otherwise}.
\end{cases}
\]
Its two marginal distributions are
\begin{align*}
\sum_{j\geq1}R(\ell,k,j)
&=
N(\ell,k)
=
\frac{1}{\ell}
\binom{\ell}{k}
\binom{\ell}{k-1},\\
\sum_{k\geq1}R(\ell,k,j)
&=
B(\ell,j)
=
\frac{j}{2\ell-j}
\binom{2\ell-j}{\ell}.
\end{align*}
Here $N(\ell,k)$ is the Narayana number, while $B(\ell,j)$ counts Dyck
paths of semilength $\ell$ with exactly $j$ returns (cf. \cite{Deutsch98}). We set
$N(\ell,k)=0$ and $B(\ell,j)=0$ outside their natural ranges, and
$B(0,0)=1$.

\begin{theorem}\label{thm:joint-descents-minima}
Let $m,n\geq2$, $d\geq0$, and $j\geq1$. Then
\begin{align*}
&\bigl|
\{P\in\F_{m,n}:
\des(P)=d,\,
\lmin(P)=j\}
\bigr|\\
&\qquad =
R(m,d+1,j)
\binom{n-2}{m-1}
+
R(m-1,d,j-1)
\binom{n-2}{m-2}.
\end{align*}
\end{theorem}

\begin{proof}
We distinguish the two cases in the proof of
Theorem~\ref{thm:Dyck-polyomino-bijection}. Suppose first that
$h_m\geq2$. Then
\[
\mathbf a(P)
=
(h_m-2,h_{m-1}-2,\ldots,h_1-2)
\in\mathcal A_{m,n-m-1}.
\]
Since the order of the column heights is reversed, descents of $P$
become ascents of $\mathbf a(P)$, while left-to-right minima become
strict right-to-left minima. Hence
\[
\des(P)=\asc(\mathbf a(P)),
\qquad
\lmin(P)=\rlmin(\mathbf a(P)).
\]
By Lemma~\ref{lem:decorated-Dyck}, a polyomino with $\des(P)=d$ and
$\lmin(P)=j$ corresponds to a Dyck path of semilength $m$ with $d+1$
peaks and $j$ returns. There are $R(m,d+1,j)$ such paths. Since the
labels have total weight $n-m-1$, each path admits
$\binom{n-2}{m-1}$ labelings. This gives the first term.

Now suppose that $h_m=1$. Then
\[
\mathbf a(P)
=
(h_{m-1}-2,\ldots,h_1-2)
\in\mathcal A_{m-1,n-m}.
\]
Since $h_{m-1}\geq2$, the last pair of columns contributes a descent,
and the last column contributes a new left-to-right minimum. Therefore,
\[
\des(P)=\asc(\mathbf a(P))+1,
\qquad
\lmin(P)=\rlmin(\mathbf a(P))+1.
\]
By Lemma~\ref{lem:decorated-Dyck}, the corresponding Dyck path has $d$
peaks and $j-1$ returns. There are $R(m-1,d,j-1)$ such paths. Since the
labels have total weight $n-m$, each path admits
$\binom{n-2}{m-2}$ labelings. This gives the second term.
\end{proof}

Summing Theorem~\ref{thm:joint-descents-minima} over $j$ gives the
descent distribution.

\begin{coro}\label{cor:descent-distribution}
For $m,n\geq2$ and $d\geq0$,
\[
\bigl|
\{P\in\F_{m,n}:\des(P)=d\}
\bigr|
=
N(m,d+1)\binom{n-2}{m-1}
+
N(m-1,d)\binom{n-2}{m-2}.
\]
For $m=1$, the unique polyomino in $\F_{1,n}$ has no descents.
\end{coro}

Similarly, summing over $d$ gives the distribution of left-to-right
minima.

\begin{coro}\label{cor:minima-distribution}
For $m,n\geq2$ and $j\geq1$,
\[
\bigl|
\{P\in\F_{m,n}:\lmin(P)=j\}
\bigr|
=
B(m,j)\binom{n-2}{m-1}
+
B(m-1,j-1)\binom{n-2}{m-2}.
\]
For $m=1$, the unique polyomino in $\F_{1,n}$ has one
left-to-right minimum.
\end{coro}

\subsection{Levels and a refinement by Motzkin numbers}
\label{subsec:levels}
Let $P$ be a Fibonacci polyomino with column heights
$(h_1,\ldots,h_m)$. Define
\[
\lev(P)
=
\bigl|
\{i:1\leq i<m,\ h_i=h_{i+1}\}
\bigr|.
\]
In both cases of Theorem~\ref{thm:Dyck-polyomino-bijection}, the passage
from the column heights to $\mathbf a(P)$ preserves equal adjacent
entries. When $h_m=1$, the omitted last column cannot form a level,
since $h_{m-1}\geq2$. Hence $\lev(P)=\lev(\mathbf a(P))$.  By Lemma~\ref{lem:decorated-Dyck}, if
$\Phi(\mathbf a(P))=(\pi,\lambda)$, then
$\lev(P)=\lv(\pi,\lambda)$. Recall that $\lv(\pi,\lambda)$ also counts the factors $UU\sigma DD$ of $\pi$, where $\sigma$ is a possibly empty Dyck path and the second
up step is labeled $0$. This description suggests considering the
corresponding statistic on the underlying unlabeled path. For a Dyck path $\pi$, let $\operatorname{pot}(\pi)$ denote the number
of factors of the form $UU\sigma DD$, 
where $\sigma$ is a possibly empty Dyck path. We call these factors
\emph{potential levels}, since such a factor contributes to
$\lv(\pi,\lambda)$ precisely when its second up step has label $0$.

We refine the notation introduced above by letting $R(\ell,k,j,e)$
denote the number of Dyck paths of semilength $\ell$ with exactly $k$
peaks, $j$ returns, and $e$ potential factors. Thus
\[
R(\ell,k,j)
=
\sum_{e\geq0}R(\ell,k,j,e).
\]
Under the classical bijection between Dyck paths and rooted plane trees, peaks correspond to leaves, returns to the degree of the root, and
potential factors to non-root vertices of outdegree $1$. Therefore, the
enumeration obtained by Kirkpatrick and Onyeze
\cite[Theorem~4.1.4]{KirkpatrickOnyeze} gives
\[
R(\ell,k,j,e)
=
\frac{j}{\ell}
\binom{\ell}{k}
\binom{\ell-k}{e}
\binom{k-j-1}{\ell-k-e-1}.
\]
Here and below, we use the convention $\binom{-1}{-1}=1$,  while every other binomial coefficient with a negative upper or lower
entry, or with lower entry larger than its upper entry, is equal to
zero. In particular, this convention includes the boundary case $\ell-e=k=j$.

Summing over the numbers of peaks and returns yields a classical
refinement of the Motzkin numbers. Under the Dyck path--plane tree
bijection, potential levels correspond to non-root vertices of
outdegree $1$, or equivalently to non-root vertices of degree $2$.
Donaghey and Shapiro \cite[(M8), Eq.~(19)]{DonagheyShapiro}
show that deleting all such vertices from a rooted plane tree produces
a branch-reduced tree, or \emph{bush}, and that bushes with $k$ edges
are counted by $M_{k-1}$. Consequently,
\[
\sum_{k,j\geq1}R(\ell,k,j,e)
=
\binom{\ell-1}{e}M_{\ell-e-1},
\qquad
0\leq e\leq\ell-1.
\]
Indeed, after deleting the $e$ non-root vertices of outdegree $1$, one
obtains a bush with $\ell-e$ edges, and the $e$ deleted vertices can be
reinserted along its edges in $\binom{\ell-1}{e}$
ways. 

The labels provide a further refinement of this distribution. For
$\ell\geq1$, $r\geq0$, and $d,j,t\geq0$, define
\[
S_{\ell,r}(d,j,t)
=
\bigl|
\{\mathbf a\in\mathcal A_{\ell,r}:
\asc(\mathbf a)=d,\,
\rlmin(\mathbf a)=j,\,
\lev(\mathbf a)=t\}
\bigr|.
\]
We extend this notation by setting $S_{\ell,r}(d,j,t)=0$ whenever one
of the parameters lies outside its natural range; in particular, this
applies when $r<0$, $d<0$, $j<0$, or $t<0$.

\begin{proposition}\label{prop:sequence-three-statistics}
For $\ell\geq1$, $r\geq0$, and $d,j,t\geq0$,
\[
S_{\ell,r}(d,j,t)
=
\sum_{e=t}^{\ell-1}
R(\ell,d+1,j,e)
\binom{e}{t}
\binom{r+\ell-e-1}{\ell-t-1}.
\]
\end{proposition}

\begin{proof}
Fix a Dyck path with $d+1$ peaks, $j$ returns, and $e$ potential
factors. To obtain exactly $t$ levels, choose the $t$ corresponding up
steps whose labels are $0$; the remaining $e-t$ distinguished labels
must be positive. After subtracting $1$ from each of these positive
labels, the remaining weight is $r-(e-t)$, distributed among
$\ell-t$ nonnegative labels. Hence the number of admissible labelings is
\[
\binom{e}{t}
\binom{r-(e-t)+(\ell-t)-1}{\ell-t-1}
=
\binom{e}{t}
\binom{r+\ell-e-1}{\ell-t-1}.
\]
Summing over $e$ gives the result.
\end{proof}

\begin{theorem}\label{thm:three-statistics-polyominoes}
Let $m,n\geq2$, $d,t\geq0$, and $j\geq1$. Then
\begin{align*}
&\bigl|
\{P\in\F_{m,n}:
\des(P)=d,\,
\lmin(P)=j,\,
\lev(P)=t\}
\bigr|\\
&\qquad =
S_{m,n-m-1}(d,j,t)
+
S_{m-1,n-m}(d-1,j-1,t).
\end{align*}
\end{theorem}

\begin{proof}
If $h_m\geq2$, then $\mathbf a(P)\in\mathcal A_{m,n-m-1}$, and
\[
\des(P)=\asc(\mathbf a(P)),
\qquad
\lmin(P)=\rlmin(\mathbf a(P)),
\qquad
\lev(P)=\lev(\mathbf a(P)).
\]
This gives the first term.

If $h_m=1$, then $\mathbf a(P)\in\mathcal A_{m-1,n-m}$. The last column contributes one descent and one new left-to-right
minimum, but no level. Hence
\[
\des(P)=\asc(\mathbf a(P))+1,
\qquad
\lmin(P)=\rlmin(\mathbf a(P))+1,
\qquad
\lev(P)=\lev(\mathbf a(P)),
\]
which gives the second term.
\end{proof}

The diagonal family takes a particularly simple form. If
$P\in\F_{n,n}$, then only the case $h_n=1$ occurs and
$\mathbf a(P)\in\mathcal A_{n-1,0}$. Thus all labels are $0$, so every
potential factor contributes a level. Consequently,
\[
S_{\ell,0}(d,j,t)
=
R(\ell,d+1,j,t).
\]

\begin{coro}
For $n\geq2$, $d\geq1$, $t\geq0$, and $j\geq2$,
\begin{align*}
&\bigl|
\{P\in\F_{n,n}:
\des(P)=d,\,
\lmin(P)=j,\,
\lev(P)=t\}
\bigr|\\
&\qquad = R(n-1,d,j-1,t)\\
&\qquad =
\frac{j-1}{n-1}
\binom{n-1}{d}
\binom{n-1-d}{t}
\binom{d-j}{n-d-t-2}.
\end{align*}
\end{coro}

Summing over $j$ gives the joint distribution of descents and levels.
Equivalently, this is the marginal distribution of leaves and non-root
vertices of outdegree $1$ in plane trees; see
\cite[Corollary~4.1.9]{KirkpatrickOnyeze}.

\begin{coro}\label{cor:diagonal-descents-levels}
For $n\geq2$ and $1\leq d\leq n-1$,
\[
\bigl|
\{P\in\F_{n,n}:
\des(P)=d,\,
\lev(P)=t\}
\bigr|
=
\frac{(n-2)!}
{(n-d-t)!\,(2d+t-n)!\,t!\,(n-1-d-t)!}
\]
for $\max\{0,n-2d\}\leq t\leq n-d-1$, and the number is $0$ otherwise.
\end{coro}

Finally, summing over both descents and left-to-right minima and using
the Motzkin marginal above gives the distribution of levels.

\begin{coro}\label{cor:levels-Motzkin}
For $n\geq2$ and $0\leq t\leq n-2$,
\[
\bigl|
\{P\in\F_{n,n}:\lev(P)=t\}
\bigr|
=
\binom{n-2}{t}M_{n-2-t}.
\]
The number is $0$ for $t>n-2$. In particular,
\[
\bigl|
\{P\in\F_{n,n}:\lev(P)=0\}
\bigr|
=
M_{n-2}.
\]
\end{coro}

\section{Concluding remarks}\label{sec:concluding}

The square-lattice model studied here exhibits two complementary
enumerative structures. With respect to area, Fibonacci polyominoes are
naturally governed by the Fibonacci numbers, whereas their enumeration
by perimeter leads to Catalan numbers.
The bijection of Lemma~\ref{lem:decorated-Dyck} explains this transition
and, more generally, transfers natural statistics on column heights to
peaks, returns, and potential factors of Dyck paths. This yields the
Narayana and Motzkin refinements developed in Section~4. It would be interesting to investigate analogous staircase polyominoes
on the triangular and hexagonal lattices. 
Determining which parts of the present framework admit analogues on
these lattices is a natural direction for future work.

\section*{Acknowledgments}
J. L. Ramírez gratefully acknowledges the hospitality of the Laboratoire d'Informatique de Bourgogne (LIB), Université de Bourgogne, where part of this work was carried out. He was also partially supported by Universidad Nacional de Colombia, Project No.~64041. We are grateful to Diego Villamizar for several helpful discussions and comments on this work.

\end{document}